\documentclass[10pt,a4paper]{article}
\usepackage{amsmath,amssymb,amsthm}
\usepackage[a4paper,top=2.05cm,bottom=2.10cm,left=2.20cm,right=2.20cm,footskip=24pt]{geometry}
\IfFileExists{microtype.sty}{\usepackage{microtype}}{}
\IfFileExists{hyperref.sty}{\usepackage[hidelinks]{hyperref}
\hypersetup{pdftitle={Strict spectral supersaturation for cliques: extremal graphs and sharp thresholds},pdfauthor={Qi Wu and Yong Lu}}}{}
\allowdisplaybreaks[2]
\newtheorem{theorem}{Theorem}[section]
\newtheorem{lemma}[theorem]{Lemma}
\newtheorem{proposition}[theorem]{Proposition}
\newtheorem{corollary}[theorem]{Corollary}
\theoremstyle{remark}

\newcommand{\lam}{\lambda}
\newcommand{\al}{\alpha}
\newcommand{\one}{\boldsymbol 1}
\newcommand{\eps}{\varepsilon}
\newcommand{\NN}{N_{K_4}}
\newcommand{\LL}{L_{m,q}}
\title{Strict spectral supersaturation for cliques:
extremal graphs and sharp thresholds\thanks{Supported by the National Natural Science Foundation of China (Nos.~12371348 and 12201258) and the High-Quality Science and Technology Cultivation Project of Jiangsu Normal University (No.~JSNUGZL2026069).}}
\author{Qi Wu, Yong Lu\thanks{Corresponding author. Email addresses: \texttt{wuqimath@163.com} (Qi Wu), \texttt{luyong@jsnu.edu.cn} (Yong Lu).}\\[3pt]
\small School of Mathematics and Statistics, Jiangsu Normal University,\\[-1pt]
\small Xuzhou, Jiangsu 221116, China}
\date{}
\begin{document}
\maketitle
\begin{abstract}
For every fixed $r\ge3$ and all sufficiently large $n$, we determine the largest adjacency spectral radius of an $n$-vertex graph with fewer than $q c_r(n)$ copies of $K_{r+1}$, where $1\le q<n/r$. Here $c_r(n)$ is the number of copies created by adding one edge to a largest part of the Tur\'an graph. We also determine all extremal graphs. In most cases the extremal graph is obtained from an almost balanced complete multipartite graph by adding a star in one part. Two small values of $q$ require separate constructions, and an additional transition occurs when $r=3$ and $n\equiv2\pmod3$. Under the non-strict constraint, the unique extremal graph is obtained by adding a $q$-edge star to a largest part of the Tur\'an graph. We determine the difference between the strict and non-strict values and prove that the sharp matching threshold is $\sqrt2(r-1)/r$. The proof treats separately the ranges $q=o(n)$, $q/m\to\gamma\in(0,1)$, and $q/m\to1$.
\end{abstract}

\section{Introduction}\label{sec:results}
All graphs in this paper are finite and simple. We use $\lam(G)$ to denote the adjacency spectral radius of a graph $G$, and use $N_F(G)$ to denote the number of copies of $F$ in $G$. Copies of a clique are counted by their vertex sets. Fix $r\ge3$ and write
\[
 n=rm+s,\qquad 0\le s<r,\qquad
 \al_r=\frac{r-1}{r},\qquad f_r(c)=2c+\frac{c^2}{\al_r}.
\]
For positive integers $n_1,\ldots,n_r$, denote by $K_r(n_1,\ldots,n_r)$ the complete $r$-partite graph with part sizes $n_1,\ldots,n_r$. Denote by $T_{n,r}$ the balanced complete $r$-partite graph on $n$ vertices. An edge with both ends in one part is called an internal edge, and an edge joining two different parts is called a cross-edge. Let $c_r(n)$ be the minimum number of copies of $K_{r+1}$ created by adding one edge to $T_{n,r}$. Thus
\[
 c_r(n)=\begin{cases}
 m^{r-1},&s=0,\\
 (m+1)^{s-1}m^{r-s},&s>0
 \end{cases}.
 \]
We consider $\Lambda_r^<(n,q)=\max\{\lam(G):|V(G)|=n,\ N_{K_{r+1}}(G)<q c_r(n)\}$.
Unless stated otherwise, we assume $1\le q<n/r$. Thus $q\le m-1$ when $s=0$, and $q\le m$ when $s>0$. For a nonnegative integer $j$, we use $a^{[j]}$ to denote $j$ parts of size $a$; a term with exponent $[0]$ is omitted.

Supersaturation problems ask for the number of forbidden subgraphs after the extremal bound is passed. The triangle case begins with an unpublished theorem of Rademacher, as reported and subsequently developed by Erd\H{o}s \cite{Erdos1962,Erdos1964}. Lov\'asz and Simonovits \cite{LovaszSimonovits1975,LovaszSimonovits1983} studied the corresponding problem for complete graphs. Erd\H{o}s and Simonovits \cite{ErdosSimonovits} established a general supersaturation theorem, and Simonovits \cite{Simonovits1968} developed a method for stability problems in extremal graph theory. Reiher \cite{Reiher} proved the clique density theorem. A graph is color-critical if it contains an edge whose deletion lowers its chromatic number. Mubayi \cite{Mubayi} obtained sharp counting results for color-critical graphs, and Pikhurko and Yilma \cite{PikhurkoYilma} obtained corresponding extremal-structure results.

In the spectral version, the number of edges is replaced by the largest adjacency eigenvalue. Bollob\'as and Nikiforov \cite{BollobasNikiforov} established inequalities relating clique counts to the spectral radius. Nikiforov \cite{Nikiforov2009,NikiforovSurvey} proved several spectral Tur\'an-type and stability results. For triangles, Li, Lu and Peng \cite{LiLuPeng} proved a spectral Erd\H{o}s--Rademacher theorem, and Li, Feng and Peng \cite{LiFengPengLS} proved a spectral Lov\'asz--Simonovits theorem.

Zhang \cite{ZhangWalks} obtained an exact walk-series equation for complete multipartite graphs with internal edges and used it to compare fixed Tur\'an extensions. Fang, Tait and Zhai \cite{FangTaitZhai} determined the corresponding fixed-extension spectral extremal graphs for disjoint color-critical subgraphs. Fang, Li, Lin and Ma \cite{FLLM} studied spectral supersaturation for color-critical graphs. Their Theorems~1.2 and~1.4(i) give the matching threshold in a square-root range. Their Theorem~7.2, together with the equality statement immediately following it, gives the star threshold and its clique equality case in a linear range. These results have the form $\lam(G)\ge\lam(H)\Longrightarrow N_{K_{r+1}}(G)\ge q c_r(n)$. In particular, their star result yields the corresponding non-strict extremal statement in that range.

To the best of our knowledge, the strict problem has not been determined throughout the full range $1\le q<n/r$. We use Zhang's response equation and refine several near-Tur\'an estimates from \cite{FLLM}. The new part is the exact strict classification, including the exceptional constructions, the transition constants, and the endpoint and gap estimates.

The strict problem is not obtained by simply replacing a $q$-edge star with a $(q-1)$-edge star. In most cases the graph still has $q$ internal edges, but its complete multipartite base changes. When $r=3$ and $s=2$, another graph with $q-1$ internal edges is sometimes better. There are also two special cases for small $q$. We determine all these graphs and then compare the strict value with the non-strict value and with the matching graph.

\subsection{The strict problem}
We need the following two almost balanced complete multipartite graphs:
\begin{align*}
 B^-_{n,r}&=K_r\bigl((m+1)^{[s+1]},m^{[r-s-2]},m-1\bigr)
 &&(r-s\ge2),\\
 B^+_{n,r}&=K_r\bigl(m+2,(m+1)^{[s-2]},m^{[r-s+1]}\bigr)
 &&(s\ge2).
\end{align*}
Each of them is obtained from the Tur\'an partition by changing two equal parts $(a,a)$ into $(a+1,a-1)$. The same two one-edge-deficient part-size patterns occur in the Erd\H{o}s--Rademacher stability analysis of Balogh and Clemen \cite{BaloghClemen}. Their problem is combinatorial, whereas here these patterns arise from spectral maximization under a strict clique-count constraint.
When we add a star in one part, all cross-edges are kept.

For $(r,s)\ne(3,2)$ define $E_{n,r,q}$ as follows.
When $s=0$, use $B^-_{n,r}$ and put a $q$-edge star in an $m$-part.
When $s=1$, use $B^-_{n,r}$ and put it in an $(m+1)$-part.
When $s=2$ and $r\ge4$, use $B^+_{n,r}$ with its edge in the $(m+2)$-part if $q=1$; for $q\ge2$, use $B^-_{n,r}$ with the star in an $(m+1)$-part.
When $3\le s<r$, use $B^+_{n,r}$ and put the star in an $(m+1)$-part.
For sufficiently large $n$, every stated part has at least $q+1$ vertices.

For tripartite graphs, let $S(u,v,w;p)$ mean $K_3(u,v,w)$ with a $p$-edge star in its $u$-part. When $r=3$ and $n=3m+2$, set $U_{m,q}=S(m+2,m,m;q)$ and $V_{m,q}=S(m,m+1,m+1;q-1)$. Also put $\LL=\max\{\lam(U_{m,q}),\lam(V_{m,q})\}$.
Let $M_m$ be $K_3(m+2,m,m)$ with a copy of $P_3$ inserted in its largest part and one edge inserted in either of its two $m$-parts. Here $P_3$ is a path on three vertices, so $M_m$ has three internal edges in total.

\begin{theorem}\label{thm:main}
Fix $r\ge3$. For all sufficiently large $n$ and all integers $1\le q<n/r$ the following statements hold.

If $(r,s)\ne(3,2)$, then $\Lambda_r^<(n,q)=\lam(E_{n,r,q})$,
and $E_{n,r,q}$ is the unique extremal graph up to isomorphism.

If $r=3$ and $s=2$, the unique extremal graph is $U_{m,q}$ when $q=1,2$, and it is $M_m$ when $q=3$. For $4\le q\le m$, $\Lambda_3^<(3m+2,q)=\max\{\lam(U_{m,q}),\lam(V_{m,q})\}$, and no other graph attains equality up to isomorphism. The real response functions of $U_{m,q}$ and $V_{m,q}$ have one crossing
\begin{equation}\label{eq:crossing-location}
 q_*(m)=\frac{\sqrt{41}-5}{4}\,m+
 \frac{887-19\sqrt{41}}{128\sqrt{41}}+O(m^{-1}).
\end{equation}
Within $4\le q\le m$, only $U$ is extremal below this crossing, only $V$ above it, and both at an integer crossing.
\end{theorem}

Direct counting shows that all the above graphs satisfy the strict inequality. The usual $B^-$ choices have clique count $q c_r(n)(1-m^{-2})$; the $B^+$ choice with its star in an $(m+1)$-part has count $q c_r(n)(1-(m+1)^{-2})$.
When $s=2$ and $q=1$, the edge in the $(m+2)$-part of $B^+$ gives $m^{r-1}<(m+1)m^{r-2}=c_r(n)$.
In the case $r=3,s=2$,
\begin{equation}\label{eq:exception-counts}
 \NN(U_{m,q})=qm^2,\quad
 \NN(V_{m,q})=(q-1)(m+1)^2,\quad
 \NN(M_m)=3m^2+2m+2<3m(m+1)
\end{equation}
for $m>2$. The difference between the bound and the second count is $(m+1)(m+1-q)>0$.

We also need the following two comparisons for small $q$. For $r\ge4$, $s=2$, the one-edge choice above exceeds the one-edge extension in an $(m+1)$-part of $B^-$ by
\begin{equation}\label{eq:one-edge-gap}
 \frac{2(r-3)}{rn^2}+O_r(n^{-3}).
\end{equation}
For $n=3m+2$, both the three-edge star and the triangle in the largest part of $B^+$ have smaller spectral radius than $M_m$, by
\begin{equation}\label{eq:mixed-gap}
 \frac{2}{n^2}+O(n^{-3}).
\end{equation}
We prove these two formulas in Subsection~\ref{sec:sublinear}. Thus neither the three-edge star nor the triangle gives the largest spectral radius in this case.

\subsection{The non-strict problem and the matching bound}
Let $T^*_{n,r,q}$ be obtained by inserting a $q$-edge star in a largest part of $T_{n,r}$. This star fits throughout $1\le q<n/r$, and its clique count is exactly $q c_r(n)$. Define $\Lambda_r^{\le}(n,q)=\max\{\lam(G):|V(G)|=n,\ N_{K_{r+1}}(G)\le q c_r(n)\}$.
\begin{theorem}\label{thm:closed}
Fix $r\ge3$. For all sufficiently large $n$ and every integer $1\le q<n/r$, $\Lambda_r^{\le}(n,q)=\lam(T^*_{n,r,q})$,
and $T^*_{n,r,q}$ is the unique extremal graph up to isomorphism. For every fixed $0<\eta<1$, there is $c_{r,\eta}>0$ such that, for sufficiently large $n$ and $\eta m\le q<n/r$, every other graph with $N_{K_{r+1}}(G)\le q c_r(n)$ satisfies
\begin{equation}\label{eq:closed-gap}
 \lam(T^*_{n,r,q})-\lam(G)\ge c_{r,\eta}/n.
\end{equation}
\end{theorem}
For cliques, Theorem~\ref{thm:closed} extends the non-strict extremal consequence of \cite[Theorem~7.2]{FLLM} and the uniqueness statement immediately following that theorem from the range $q\le\delta n$ for some $\delta=\delta(r)>0$ to all $1\le q<n/r$, that is, to every value for which the star fits. Notice that the non-strict inequality removes the extra graphs that appear in the strict problem when $q=3$.

Let $Y_{n,r,q}$ be obtained by inserting a matching of size $q$ in a largest part of $T_{n,r}$, whenever that part contains at least $2q$ vertices. In particular, it is defined in every fixed square-root range for large $n$, and $N_{K_{r+1}}(Y_{n,r,q})=q c_r(n)$.
\begin{theorem}\label{thm:matching-threshold}
Fix $r\ge3$. The supremum of the constants $\delta>0$ for which, for all sufficiently large $n$ and every integer $1\le q<\delta\sqrt n$,
\[
 \lam(G)\ge\lam(Y_{n,r,q})
 \quad\Longrightarrow\quad N_{K_{r+1}}(G)\ge q c_r(n)
\]
is $\delta_r^*=\sqrt2\,\frac{r-1}{r}$.
For each value of $s$, the real star and matching response roots have a unique crossing in every fixed window $a\sqrt n\le q\le b\sqrt n$ with $0<a<\delta_r^*<b$. Its location is
\begin{equation}\label{eq:matching-location}
 \theta_{r,s}(n)=\delta_r^*\sqrt n+\beta_{r,s}+O_r(n^{-1/2}),
 \qquad
 \beta_{r,s}=\begin{cases}11/6,&(r,s)=(3,2),\\[1mm]1/2,&\text{otherwise}.
 \end{cases}
\end{equation}
Here the strict star graph is $E_{n,r,q}$ when $(r,s)\ne(3,2)$ and is $U_{m,q}$ when $(r,s)=(3,2)$. The two small-$q$ exceptions are outside these windows.
\end{theorem}

\subsection{Outline of the proof}
We work with a maximum $r$-cut and write
\[
 G=K_r(n_1,\ldots,n_r)+H-D,\qquad
 H=H_1\cup\cdots\cup H_r,\quad p=e(H),\quad z=|D|,
\]
where $H_i$ lies in the $i$th part and $D$ consists of missing cross-edges. Put $S_2=\sum_i(n_i-n/r)^2$,
and let $M_2(H)$ count unordered pairs of disjoint internal edges, including pairs in different parts.

We use a maximum $r$-cut throughout the proof. First, the stability theorem of Fang, Li, Lin and Ma \cite[Theorem~2.8]{FLLM} reduces the problem to $O(n)$ edits of a complete $r$-partite graph. Their perturbation estimates are sharpened below to retain the signed quadratic term needed here. The exact comparisons use the multipartite response equation in Section~1 of Zhang \cite{ZhangWalks}. We then consider three cases. If $q=o(n)$, pairs of disjoint internal edges force the internal graph to be a star. If $q/m$ has a limit in $(0,1)$, an integer comparison gives the result. If $q/m\to1$, we delete the vertex meeting most internal edges and study the remaining core. Section~\ref{sec:tools} gives the required estimates, Section~\ref{sec:exact} proves the exact theorems, Section~\ref{sec:crossing} studies the two transitions, and Section~\ref{sec:gaps} gives further gap estimates. Standard Schur-complement arguments are used in several local calculations; cf. \cite{StewartSun}.

\section{Notation and preliminary lemmas}\label{sec:tools}
Put $\al=(r-1)/r$. In a maximum $r$-cut every vertex has at least as many neighbors in each other part as in its own part. All estimates below are uniform when the fixed parameters are fixed.

\subsection{Completion counts and Perron coordinates}
For an internal edge $e$ in part $V_i$, a \emph{standard completion} of $e$ is obtained by choosing one vertex in every part $V_j$ with $j\ne i$. Thus, before cross-edges are deleted, $e$ has exactly $\prod_{j\ne i}n_j$ standard completions. The next lemma is a clique-specific refinement of the near-Tur\'an counting estimate in \cite[Lemma~2.5]{FLLM}; the extra point is the distinction between missing cross-edges that meet the hub and those that avoid it.

\begin{lemma}\label{lem:standard-completions}
Let
\[
 G=K_r(n_1,\ldots,n_r)+H-D,
 \qquad p_i=e(H_i),\quad p=\sum_i p_i,\quad z=|D|.
\]
Then, for a constant $C_r$ depending only on $r$,
\begin{equation}\label{eq:standard-completions}
 \begin{split}
 N_{K_{r+1}}(G)\ge{}&\sum_i p_i\prod_{j\ne i}n_j\\
 &-C_r n^{r-2}\sum_{ab\in D}\bigl(d_H(a)+d_H(b)\bigr)
 -C_rpzn^{r-3}.
 \end{split}
\end{equation}
In particular, the error in \eqref{eq:standard-completions} is at most $C_rpzn^{r-2}$. If a vertex $v$ meets all but $k$ internal edges, then one missing edge meeting $v$ destroys at most $C_rpn^{r-2}$ standard completions, while one missing edge avoiding $v$ destroys at most
\begin{equation}\label{eq:standard-completions-hub}
 C_r\bigl((k+1)n^{r-2}+pn^{r-3}\bigr)
\end{equation}
standard completions.
\end{lemma}
\begin{proof}
The standard-completion families of distinct internal edges are disjoint: a clique in one such family has a unique part containing two vertices, and those two vertices recover the internal edge. Fix a missing cross-edge $ab$. If a standard completion containing $ab$ is assigned to an internal edge incident with $a$ or $b$, then after that internal edge and $ab$ are fixed there are at most $n^{r-2}$ choices. There are at most $d_H(a)+d_H(b)$ such internal edges. For every other internal edge there are at most $n^{r-3}$ choices, since both ends of $ab$ are then fixed among the vertices chosen outside the doubled part. A union bound over $D$ proves \eqref{eq:standard-completions}.

The crude bound follows from $d_H(a)+d_H(b)\le2p$. If all but $k$ internal edges meet $v$, every vertex other than $v$ has internal degree at most $k+1$. Applying the first estimate separately to missing edges meeting and avoiding $v$ gives \eqref{eq:standard-completions-hub}.
\end{proof}

We shall also use the following coordinate estimate. The flat-coordinate estimate in \cite[Claim~4.3]{FLLM} is sufficient for unweighted near-Tur\'an graphs. We need a weighted version and a separate estimate at a vertex carrying linearly many internal edges. An internal edge of weight in $[0,1]$ is counted with its weight, and the same convention is used for a cross-edge whose weight has been lowered from one.

\begin{lemma}\label{lem:perron-coordinates}
Fix $r\ge3$. Let $A_\omega$ be obtained from $K_r(n_1,\ldots,n_r)$ by adding internal edges of total weight $P=O(n)$ and lowering the weights of a set of $Z=o(n)$ cross-edges. Suppose
\[
 \max_i|n_i-n/r|\le A=o(n),\qquad
 \lam(A_\omega)=\al n+O(1).
\]
For all sufficiently large $n$, the weighted graph is connected. Let $x$ be its positive unit Perron vector and put $M=\max_u x_u$.
Then $M=O(n^{-1/2})$. If $E_0=A+P+Z+1=o(n)$, then
\begin{equation}\label{eq:flat-coordinates}
 x_u=\frac{1+O(E_0/n)}{\sqrt n}
\end{equation}
uniformly in $u$.

Suppose instead that a vertex $v$ is incident with internal edges of total weight $P_v$, and that the total weight of the other internal edges is at most $K=o(n)$. Put $E=A+K+Z+1$. Then, uniformly when $E=o(n)$,
\begin{equation}\label{eq:hub-coordinates}
 x_u=\frac{1+O(E/n)}{\sqrt n}\quad(u\ne v),
 \qquad
 x_v=\frac{1+P_v/(\al n)+O(E/n)}{\sqrt n}.
\end{equation}
The constants are uniform over all edge weights in $[0,1]$.
\end{lemma}
\begin{proof}
Since every part has size $\Theta(n)$ and only $Z=o(n)$ cross-edge weights are lowered, every vertex has positive-weight degree at least $n-n_i-Z>(1/2+o(1))n$. Hence the positive-weight skeleton, and therefore the weighted graph, is connected for large $n$. The eigenvector equation and Cauchy--Schwarz give
$\lam(A_\omega)x_u\le\sqrt n$ for every $u$, and hence $M=O(n^{-1/2})$. Write
\[
 S_i=\sum_{u\in V_i}x_u,\qquad S=\sum_iS_i.
\]
On summing the eigenvector equations over $V_i$, every edited edge contributes the sum of its two endpoint coordinates, with the appropriate sign. Therefore
\begin{equation}\label{eq:part-sums}
 (\lam+n_i)S_i=n_iS+R_i,
 \qquad |R_i|\le2(P+Z)M=O(\sqrt n).
\end{equation}
In particular,
\begin{equation}\label{eq:part-sums-2}
 S_i=\frac{n_iS}{\lam+n_i}+O(n^{-1/2}).
\end{equation}

Assume first that $E_0=A+P+Z+1=o(n)$. In this case the sharper form of \eqref{eq:part-sums} is $|R_i|=O((P+Z)n^{-1/2})$, and hence
$S_i=n_iS/(\lam+n_i)+O((P+Z)n^{-3/2})$. For every vertex $u\in V_i$, the edit contribution in its eigenvector equation has absolute value at most $(P+Z)M$. Equations \eqref{eq:part-sums-2} and $\lam+n_i=n+O(A+1)$ therefore give
\[
 x_u=\frac Sn+O(E_0n^{-3/2})
\]
uniformly. The normalization condition gives $S=\sqrt n(1+O(E_0/n))$, proving \eqref{eq:flat-coordinates}.

For the second assertion, every vertex other than $v$ has internal degree at most $K+1$. Thus, for $u\in V_i\setminus\{v\}$, \eqref{eq:part-sums-2} and the individual eigenvector equation give
\begin{equation}\label{eq:nonhub-coordinate-proof}
 x_u=\frac{S}{\lam+n_i}+O(En^{-3/2})
 =\frac Sn\left(1+O\left(\frac{A+1}{n}\right)\right)+O(En^{-3/2}).
\end{equation}
Let $v\in V_1$. Summing \eqref{eq:nonhub-coordinate-proof} over its internal neighbors, with edge weights included, and then using the equation at $v$, gives
\begin{equation}\label{eq:hub-coordinate-proof}
 x_v=\frac Sn\left(1+\frac{P_v}{\al n}+O\left(\frac En\right)\right)+O(En^{-3/2}).
\end{equation}
Since $S\le\sqrt n$, equations \eqref{eq:nonhub-coordinate-proof} and \eqref{eq:hub-coordinate-proof}, followed by $\sum_ux_u^2=1$, yield
$S=\sqrt n(1+O(E/n))$. Substitution proves \eqref{eq:hub-coordinates}.
\end{proof}

\subsection{Stability and limit neighborhoods}
Let $A_0$ be the adjacency matrix of a complete $r$-partite graph, let $\mu$ be its largest eigenvalue, and let $x$ be its unit positive eigenvector. If $E$ is symmetric and $\|E\|=o(\mu)$, the Schur complement on $x^\perp$ gives
\begin{equation}\label{eq:schur-upper}
 \lam(A_0+E)\le \mu+x^{\mathsf T}Ex+
 \frac{\|Ex\|^2}{\mu-2\|E\|}.
\end{equation}
The other eigenvalues of $A_0$ are nonpositive, so the denominator is valid.

The following statement refines the first-order perturbation estimate in \cite[Theorem~1.5]{FLLM}. The term $\sum_v s_v^2$ is retained because it separates the competing linear constructions.

\begin{lemma}\label{lem:linear-signed}
Fix $r\ge3$, $C_0>0$ and $L_0\ge0$. If $N_{K_{r+1}}(G)\le C_0n^r$ and $\lam(G)\ge\al n-L_0$, then every maximum $r$-cut satisfies $p+z=O(n)$ and $S_2=O(n)$.
Moreover, with $s_v=d_H(v)-d_D(v)$,
\begin{equation}\label{eq:signed-expansion}
 \lam(G)=\al n-\frac{\al S_2}{n}
 +\frac{2(p-z)}n+\frac1{\al n^2}\sum_vs_v^2+o(1).
\end{equation}
\end{lemma}
\begin{proof}
Since $C_0n^r=o(n^{r+1})$, the spectral supersaturation-stability theorem of Fang, Li, Lin and Ma \cite[Theorem~2.8]{FLLM} gives a partition for which $G$ differs from $T_{n,r}$ in $o(n^2)$ edges. A maximum $r$-cut has no more internal edges than this partition. Also $e(G)=e(T_{n,r})+o(n^2)$. Hence, for the chosen maximum cut, $p+z=o(n^2),\qquad n_i=n/r+o(n)$.
Indeed,
\[
 \sum_{i<j}n_in_j\ge e(G)-p=e(T_{n,r})-o(n^2),
 \qquad
 \sum_i(n_i-n/r)^2=\al n^2-2\sum_{i<j}n_in_j.
\]

Put $\theta=1/(100r)$ and $\xi=\theta/1000$. Let $X$ be the set of vertices with internal degree at least $\xi n$. If $v\in X$, then the maximum-cut property gives at least $\xi n$ neighbors of $v$ in every part. We choose an internal neighbor of $v$ and one neighbor in each other part. This gives at least $(\xi n)^r$ choices. At most $zn^{r-2}=o(n^r)$ of them are lost because of missing cross-edges. Thus every $v\in X$ lies in at least $\xi^rn^r/2$ copies of $K_{r+1}$ for large $n$. Since one copy contains at most $r+1$ vertices of $X$, double counting the incidences between $X$ and these copies gives $|X|\le 2(r+1)C_0\xi^{-r}=O(1)$.

Delete $X$ and write $N=n-|X|$. Let $p',z'$ be the two edit counts in $G-X$. The remaining maximum internal degree is at most $\beta N$, where $\beta=2\xi$. If $y$ is a unit Perron vector of $G$, then $y_v\le\sqrt n/\lam(G)=O(n^{-1/2})$. Put $a=\sum_{v\in X}y_v^2=O(n^{-1})$. Using the restriction of $y$ to $G-X$ and the eigenvalue equations on $X$, we obtain $\lam(G-X)\ge \lam(G)\frac{1-2a}{1-a}\ge \lam(G)-O(1)$.
Hence
\begin{equation}\label{eq:delete-X}
 \lam(G-X)\ge \al N-L_1
\end{equation}
for a fixed $L_1$.

Let $a_v,b_v$ be the internal and missing-cross degrees in $G-X$, and put $Y=\{v:b_v\ge\theta N\}$.
Then $|Y|\le2z'/(\theta N)=o(N)$. Consider an internal edge $uv$ with both ends outside $Y$. Before cross-edges are removed it has $(r^{-1}+o(1))^{r-1}N^{r-1}$ standard completions. Missing edges incident with $u$ or $v$ destroy at most $2\theta N^{r-1}$ of them, and all other missing edges destroy at most $O_r(z'N^{r-3})=o(N^{r-1})$ of them. Taking $\theta$ sufficiently small, every such edge has at least $c'_rN^{r-1}$ standard completions. These families are disjoint by Lemma~\ref{lem:standard-completions}. Internal edges meeting $Y$ number at most $\beta N|Y|$. Hence
\begin{equation}\label{eq:count-contraction-short}
 p'\le C_1N+\frac{2\beta}{\theta}z'.
\end{equation}

Let $A_0$ be the complete multipartite base on the same parts and $E=A(G-X)-A_0$. Its unit Perron coordinates are $(1+o(1))/\sqrt N$, because a coordinate in part $i$ is proportional to $(\mu+n_i)^{-1}$. Since $\|E\|\le\sqrt{2(p'+z')}=o(N)$, $x^{\mathsf T}Ex=\frac{2(p'-z')+o(p'+z')}{N}$.
Moreover,
\[
 \|Ex\|^2\le\frac1N\left(\sum_v(a_v-b_v)^2+o\left(\sum_v(a_v+b_v)^2\right)\right).
\]
We have $\sum_va_v^2\le2\beta Np'$. Vertices outside $Y$ contribute at most $2\theta Nz'$ to $\sum_vb_v^2$. On $Y$, $\sum_{v\in Y}b_v\le z'+e_D(Y)\le z'+\tfrac12|Y|^2=z'+o(z')$,
and $b_v\le(\al+o(1))N$. Thus $\sum_vb_v^2\le(\al+2\theta+o(1))Nz'$.
Using $(a_v-b_v)^2\le a_v^2+b_v^2$ in \eqref{eq:schur-upper}, and then taking $n$ large, gives $\lam(G-X)\le \al N+\frac{3p'}N-\frac{z'}{2N}$.
Together with \eqref{eq:delete-X}, this yields $z'\le6p'+2L_1N$. Putting this in \eqref{eq:count-contraction-short}, with $12\beta/\theta<1/2$, gives $z'=O(N)$ and then $p'=O(N)$. Restoring the bounded set $X$ adds only $O(n)$ edits, so $p+z=O(n)$.

Now we prove $S_2=O(n)$. For the complete multipartite base, if $F_0(x)=\sum_i n_i/(x+n_i)$ and $\mu$ is its root, then $1-F_0(\al n)=\frac{\al}{n}\sum_i\frac{(n_i-n/r)^2}{n+n_i-n/r}$.
Since $|F_0'(x)|=O(n^{-1})$ when $x=\Theta(n)$, this implies $\al n-\mu\ge c''_rS_2/n$.
Now $p+z=O(n)$, $|x^{\mathsf T}Ex|=O(1)$, and $\|Ex\|^2=O(n)$. Formula \eqref{eq:schur-upper} gives $\lam(G)\le\mu+O(1)$; hence $S_2=O(n)$.

It remains to prove the expansion. Let $Q=I-xx^{\mathsf T}$. The Schur formula and the identity for the inverse give
\[
 \lam(G)=\mu+x^{\mathsf T}Ex+
 \langle QEx,(\mu I-QA_0Q)^{-1}QEx\rangle
 +O\left(\frac{\|E\|\,\|Ex\|^2}{n^2}\right).
\]
Here $\|E\|=O(\sqrt n)$ and $\|Ex\|^2=O(n)$, so the remainder is $O(n^{-1/2})$. The base coordinate in part $i$ is $n^{-1/2}(1+O(n^{-1/2}))$, so
\[
 x^{\mathsf T}Ex=\frac{2(p-z)}n+O(n^{-1/2}),
 \qquad
 \|Ex\|^2=\frac1n\sum_vs_v^2+O(\sqrt n).
\]
On the subspace whose coordinate sum in each part is zero, $A_0=0$ and the inverse is $\mu^{-1}I$. The projection of $Ex$ onto the remaining part-constant directions has norm $O(1)$, so those directions contribute only $O(n^{-1})$. Therefore the quadratic term is $\frac1{\al n^2}\sum_vs_v^2+O(n^{-1/2})$.
Expanding the base equation $\sum_i n_i/(\mu+n_i)=1$ gives
$\mu=\al n-\al S_2/n+O(n^{-1/2})$, which proves \eqref{eq:signed-expansion}.
\end{proof}

For $0\le\gamma\le1$, define $\phi_r(\gamma)=2\gamma/r+\gamma^2/[r(r-1)]=f_r(\gamma/r)$.

\begin{theorem}\label{thm:profiles}
Let $n=rm+s\to\infty$, $1\le q\le m$, and $q/m\to\gamma$. Suppose $N_{K_{r+1}}(G)\le q c_r(n)$ and $\lam(G)\ge\al n-O(1)$. Then $\limsup(\lam(G)-\al n)\le\phi_r(\gamma)$.
If in addition $\lam(G)\ge\al n+\phi_r(q/m)-o(1)$, then every maximum $r$-cut satisfies
\begin{equation}\label{eq:limit-entry}
 z=o(n),\qquad p=q+o(n),\qquad S_2=o(n).
\end{equation}
If $\gamma>0$, all but $o(n)$ internal edges meet one vertex.
\end{theorem}
\begin{proof}
Lemma~\ref{lem:linear-signed} applies. Fix an arbitrary sequence of maximum cuts; then $p+z=O(n)$ and $S_2=O(n)$. Order the vertices by decreasing edit degree, say $v_1,v_2,\ldots$, and pass to a subsequence and use a diagonal argument so that $q/m$, $S_2/n$, $p/m$, $z/m$, $\lam(G)-\al n$, and all the ratios of neighborhood sizes below converge. If $v_j$ lies in part $i_j$, set
\[
 x_{j,i}=\lim\frac{|N(v_j)\cap V_i|}{m},\qquad
 a_j=x_{j,i_j},\qquad
 b_j=\sum_{i\ne i_j}(1-x_{j,i}),
\]
Put $t_j=a_j-b_j=\sum_i x_{j,i}-(r-1)$ and $\pi_j=\prod_i x_{j,i}$.
The maximum-cut condition gives $x_{j,i_j}\le x_{j,i}$ for every $i$.

Delete $v_1,\ldots,v_K$, and let $p_K,z_K$ be the remaining edit counts. Edits with both ends among the deleted vertices contribute only $O(K^2)$, so for every fixed $K$,
\[
 \lim\frac{p_K}{m}=\lim\frac pm-\sum_{j\le K}a_j\ge0,
 \qquad
 \lim\frac{z_K}{m}=\lim\frac zm-\sum_{j\le K}b_j\ge0.
\]
Hence $\sum_j(a_j+b_j)<\infty$. Put
\begin{equation}\label{eq:limit-remaining}
 P=\lim p/m-\sum_ja_j\ge0,
 \qquad
 Z=\lim z/m-\sum_jb_j\ge0.
\end{equation}
The numbers $P$ and $Z$ count the edges that are not seen at any fixed number of vertices.

Let $d_*(v)=d_H(v)+d_D(v)$. Since $d_*(v_{K+1})\le 2(p+z)/(K+1)$ and $\sum_vd_*(v)=2(p+z)$, $\frac1{m^2}\sum_{j>K}s_{v_j}^2 \le \frac{C_r}{K+1}$.
Writing $S_2/n\to\sigma$ and first taking $n\to\infty$, then $K\to\infty$, formula \eqref{eq:signed-expansion} gives
\begin{equation}\label{eq:limit-spectrum-short}
 \Lambda=-\al\sigma+\frac2r\left(P-Z+\sum_jt_j\right)
 +\frac1{r(r-1)}\sum_jt_j^2,
\end{equation}
where $\Lambda=\lim(\lam(G)-\al n)$.

We now use the clique bound. Fix $K$ and $j\le K$. Count the cliques containing $v_j$, no other vertex among $v_1,\ldots,v_K$, and one neighbor of $v_j$ in every part. Divided by $m^r$, their number tends to $\pi_j$; missing cross-edges among the other selected vertices spoil only $O_r(zn^{r-2})=O(n^{r-1})$ choices. These $K$ families are disjoint. Every internal edge outside the first $K$ vertices has $\bigl(1-O_r(K^{-1})-o(1)\bigr)m^{r-1}$
standard completions avoiding those vertices, because every remaining edit degree is $O(n/K)$. Missing cross-edges incident with an endpoint of the internal edge remove $O_r(n^{r-1}/K)$ choices, and all other missing cross-edges remove $O_r(zn^{r-3})=O(n^{r-2})$ choices. Thus the same displayed lower bound remains valid. The resulting families are disjoint from one another and from the first $K$ families by Lemma~\ref{lem:standard-completions}. Since $qc_r(n)/m^r\to\gamma$, taking the limits in the order $n\to\infty$, then $K\to\infty$, yields
\begin{equation}\label{eq:limit-budget-short}
 P+\sum_j\pi_j\le\gamma.
\end{equation}

For $x_1,\ldots,x_r\in[0,1]$,
\begin{equation}\label{eq:product-bound}
 \prod_i x_i\ge\sum_i x_i-(r-1).
\end{equation}
Indeed, with $y_i=1-x_i$, expand $1-\prod_i(1-y_i)$ one by one. Equality holds exactly when at most one $x_i$ is smaller than one. Thus $\pi_j\ge(t_j)_+$. Put $T=\sum_j(t_j)_+, \qquad U=\sum_j(-t_j)_+$.
For $-(r-1)\le t<0$, $\frac{2t}{r}+\frac{t^2}{r(r-1)}\le\frac tr$.
Since $P+T\le\gamma$, equations \eqref{eq:limit-spectrum-short} and \eqref{eq:limit-budget-short} give
\begin{align}
 \Lambda
 &\le-\al\sigma-\frac{2Z}{r}-\frac Ur
 +\frac{2(P+T)}r+\frac{\sum_j(t_j)_+^2}{r(r-1)}\notag\\
 &\le-\al\sigma-\frac{2Z}{r}-\frac Ur
 +\frac{2(P+T)}r+\frac{T^2}{r(r-1)}\notag\\
 &\le\phi_r(\gamma).\label{eq:limit-chain}
\end{align}
The chain \eqref{eq:limit-chain} proves the upper bound.

Under the stronger lower bound, equality must hold in every step. Hence $\sigma=Z=U=0, \qquad P+T=\gamma, \qquad \sum_j(t_j)_+^2=\gamma^2$.
If $\gamma>0$, then $P=0$ and exactly one number $t_j$ is positive; its value is $\gamma$. If $\gamma=0$, then $P=T=0$. Equality in the count bound gives $\pi_j=t_j\ge0$ for every $j$. By the equality case of \eqref{eq:product-bound}, for each $j$ all but at most one of the numbers $x_{j,i}$ are equal to one. The maximum-cut condition shows that the possible smaller number is $x_{j,i_j}$. Thus $b_j=0$ and $a_j=t_j$ for every $j$. Equation \eqref{eq:limit-remaining} gives $z=o(n)$ and $p=q+o(n)$, and $\sigma=0$ gives $S_2=o(n)$. When $\gamma>0$, the vertex with positive $t_j$ has internal degree $\gamma m+o(n)=p-o(n)$.

The same argument applies to every subsequence. Hence the conclusions hold for the original sequence and for every maximum cut.
\end{proof}

\begin{corollary}\label{cor:global-entry}
Fix $r\ge3$, $L\ge0$ and $\eps>0$. For all sufficiently large $n=rm+s$, every graph with $1\le q\le m$, $N_{K_{r+1}}(G)\le q c_r(n)$ and $\lam(G)\ge\al n+f_r(q/n)-L/n$ has a maximum $r$-cut with $p\le n$ and $z\le\eps n$.
\end{corollary}
\begin{proof}
Suppose to the contrary that the corollary is false. Take a counterexample sequence and pass to a subsequence with $q/m\to\gamma$. Since
$\phi_r(q/m)=f_r(q/n)+O_r(n^{-1})$, the spectral hypothesis meets the equality case assumption of Theorem~\ref{thm:profiles}. Equation~\eqref{eq:limit-entry} gives $z=o(n)$ and $p=q+o(n)<n$.
\end{proof}

The floor-plus-square term below is the sparse quasi-star profile familiar from the classical problem of maximizing the sum of squared degrees; see \cite{AhlswedeKatona,AbregoEtAl}. Here it is derived from the strict clique budget and the near-Tur\'an spectral structure.

\begin{theorem}\label{thm:global-asymptotic}
Fix $r\ge3$ and $\Gamma>0$. For $\gamma\ge0$, set $j=\lfloor\gamma\rfloor$, $\vartheta=\gamma-j$, and $\Phi_r(\gamma)=2\gamma/r+(j+\vartheta^2)/[r(r-1)]$.
Uniformly for $1\le q\le\Gamma\lfloor n/r\rfloor$,
\begin{equation}\label{eq:general-envelope}
 \max\{\lam(G):|V(G)|=n,\ N_{K_{r+1}}(G)<q c_r(n)\}
 =\al n+\Phi_r(q/m)+o(1).
\end{equation}
\end{theorem}
\begin{proof}
The preceding argument remains valid whenever $q/m$ has a bounded limit $\gamma$. Since each positive $t_j\le1$, $\sum_j(t_j)_+^2\le\lfloor T\rfloor+(T-\lfloor T\rfloor)^2$.
For a finite sequence, if two entries lie strictly between zero and one, moving mass from the smaller entry to the larger one until one entry becomes zero or one does not decrease the sum of squares. Iteration proves the displayed inequality; the infinite case follows by truncation. The right side is increasing in $T$. Together with $P+T\le\gamma$, this gives the upper bound.

For the lower bound, the case $q<m$ is given by a $(q-1)$-edge star in a largest Tur\'an part. If $q=jm+s_0\ge m$, where $0\le s_0<m$, take a balanced complete $r$-partite core on $n-j-1$ vertices. Add $j$ vertices adjacent to every other vertex and one further vertex $v$, joined to every core vertex outside the first part and to $\min\{s_0,b_1\}$ vertices in that part; assign all added vertices to the first part. Then $p=q+O_{r,\Gamma}(1)$ and $\sum_ud_H(u)^2=jm^2+s_0^2+O_{r,\Gamma}(n)$, so \eqref{eq:fine} gives spectral radius
$\al n+\Phi_r(q/m)+O_{r,\Gamma}(n^{-1})$. Its clique count is
$qc_r(n)+O_{r,\Gamma}(n^{r-1})$. Deleting a sufficiently large bounded number of edges from one universal vertex to the first core part makes the inequality strict; each deletion destroys $\Theta_r(n^{r-1})$ cliques and changes the root by $O(n^{-1})$. Finally, if the error were not uniform, a counterexample sequence would have a subsequence for which $q/m$ converges in $[0,\Gamma]$. This would contradict the upper bound already proved. Hence the error in \eqref{eq:general-envelope} is uniform.
\end{proof}

\subsection{Response equations and stars}
For an internal graph $J$ on $u$ vertices, including isolated vertices, define $R_J(x)=\one^{\mathsf T}(xI-A(J))^{-1}\one$ for $x>\lam(J)$. Since
$R_J(x)=u/x+\sum_{\ell\ge1}w_\ell(J)/x^{\ell+1}$, The exact multipartite walk-series identity of Zhang \cite{ZhangWalks}, written in resolvent form, states that the spectral radius is the unique root of
\begin{equation}\label{eq:response}
 F(x)=\sum_{i=1}^r\frac{R_{H_i}(x)}{1+R_{H_i}(x)}=1.
\end{equation}
The uniqueness also follows from $R_J'(x)=-\one^{\mathsf T}(xI-A(J))^{-2}\one<0$. For a $p$-edge star in a part of size $u$, direct evaluation gives
\begin{equation}\label{eq:star-response}
 R_{\rm st}(x;u,p)=\frac ux+
 \frac{p(2x+p+1)}{x(x^2-p)}.
\end{equation}

For a fixed number of edges, Ismailescu and Stefanica \cite[Lemma~2.4]{IsmailescuStefanica} determined the graphs maximizing $w_2=\sum_v d(v)^2$ in the sparse range: the star is unique except when $p=3$, where the star and the triangle tie. Zhang \cite{ZhangWalks} observed that the third walk count selects the triangle when $p=3$ and the star otherwise. We need the following quantitative form away from the extremal graph.

\begin{lemma}\label{lem:walks}
Let $J$ have $p$ edges and let $w_j=\one^{\mathsf T}A(J)^j\one$. Then $w_1=2p,\qquad w_2=p^2+p-2M_2(J),\qquad w_3\le4p^2$,
and $w_4\le2pw_2$. If $p\ge4$ and $J$ is not a star together with isolated vertices, then $M_2(J)\ge p-3$. If $p\ge16$ and $k=p-\Delta(J)$, then $pk/16\le M_2(J)\le pk$.
\end{lemma}
\begin{proof}
Counting intersecting edge pairs gives $w_2=2p+2\left(\binom p2-M_2(J)\right)$.
If $d=A(J)\one$, then
\[
 w_3=2\sum_{ab\in E(J)}d(a)d(b)
 \le\left(\sum_ad(a)\right)^2=4p^2,
 \qquad
 w_4=\|A(J)d\|^2\le2p\|d\|^2,
\]
because $\|A(J)\|\le\sqrt{2p}$.

Choose a vertex of maximum degree $d_0=p-k$. Every edge avoiding it is disjoint from at least $d_0-2$ incident edges. If $d_0\ge3$ and $k\ge1$, then $M_2(J)\ge k(p-k-2)\ge p-3$;
the last inequality follows by concavity on $1\le k\le p-3$. If $d_0\le2$, every edge meets at most two other edges, so
$M_2(J)\ge\binom p2-p\ge p-3$. A pairwise-intersecting family of at least four edges has a common endpoint, which proves the stated description of the star case.

Every disjoint pair contains an edge avoiding the chosen vertex, hence $M_2(J)\le pk$. If $M_2(J)\ge p^2/16$, then $M_2(J)\ge pk/16$ because $k\le p$. Otherwise $w_2=p^2+p-2M_2(J)>\frac78p^2$,
whereas $w_2=\sum_vd(v)^2\le2p\Delta(J)=2pd_0$. Thus $d_0>7p/16$ for $p\ge16$, and $d_0-2\ge p/4$. The first paragraph then gives $M_2(J)\ge k(d_0-2)\ge pk/4$.
This completes the proof.
\end{proof}

For $p\le Bn$ and $x=\Theta(n)$, the walk series gives
\begin{equation}\label{eq:R-expansion}
 R_J(x)=\frac ux+\frac{2p}{x^2}+\frac{w_2}{x^3}
 +O_B(p^2/n^4).
\end{equation}
Indeed, the length-three term is at most $4p^2/x^4$. For $j\ge4$, the spectral theorem and $w_4\le2pw_2$ give
$|w_j|\le\|A(J)\|^{j-4}w_4$, and the remaining tail is $O\left(\frac{p^3}{x^5(1-\sqrt{2p}/x)}\right)=O_B(p^2/n^4)$.
The same bounds justify differentiation on every fixed interval $x=\al n+O(1)$.

Let $K=K_r(n_1,\ldots,n_r)$, put $a_i=n_i-n/r$, $A=\max_i|a_i|$, $S_2=\sum_i a_i^2$, and let $\mu=\lam(K)$. Its base response $F_0(x)=\sum_i n_i/(x+n_i)$ satisfies
\begin{equation}\label{eq:base-variance-identity}
 1-F_0(\al n)=\frac{\al}{n}\sum_i\frac{a_i^2}{n+a_i}.
\end{equation}
Hence $\mu\le\al n$, and whenever $\mu\ge cn$ for a fixed $c>0$,
\begin{equation}\label{eq:variance-loss}
 \al n-\mu\ge c'_rS_2/n,
\end{equation}
because $|F_0'(x)|=O(n^{-1})$ on the relevant interval.

Put $p_i=e(H_i)$. Expanding \eqref{eq:response} in walks gives the following formulas.

\begin{lemma}\label{lem:response-expansions}
If $\max_i|a_i|=O(1)$ and $p=O(n)$, then
\begin{equation}\label{eq:fine}
 \lam(K+H)=\lam(K)+\frac{2p}{n}-\frac4{n^2}\sum_i a_ip_i
 +\frac{\sum_vd_H(v)^2}{\al n^2}
 +O\left(\frac{p^2+p}{n^3}\right).
\end{equation}
If $p=O(n)$, $S_2=O(n)$ and $A=\max_i|a_i|$, then
\begin{equation}\label{eq:coarse}
 \lam(K+H)=\al n+f_r(p/n)-\frac{\al S_2}{n}
 -\frac{2M_2(H)}{\al n^2}+O\left(\frac{A+1}{n}\right).
\end{equation}
\end{lemma}
\begin{proof}
We first prove the coarse formula. By the hypotheses, choose fixed constants $B_0,C_0$ such that $p\le B_0n$ and $S_2\le C_0n$. A direct use of \eqref{eq:R-expansion} and \eqref{eq:base-variance-identity}, evaluated at $x=\al n+w$, gives, uniformly for $|w|\le n^{1/4}$, $F(\al n+w)=1-\frac{w}{n}+O_{r,B_0,C_0}(n^{-1})$.
Let $C$ bound the last error and choose a fixed $C_1>2C$. Since $F$ is strictly decreasing, $F(\al n-C_1)>1>F(\al n+C_1)$
for all sufficiently large $n$. Thus the root has the form $x=\al n+w$ with $w=O(1)$. Expanding the base response at $\al n$ and using \eqref{eq:base-variance-identity} now gives
\[
 F_0(x)=1-\frac wn-\frac{\al S_2}{n^2}
 +O\left(\frac{(A+1)S_2}{n^3}+\frac1{n^2}\right).
\]
Let $S=\sum_vd_H(v)^2$. By \eqref{eq:R-expansion}, the change in the $i$th part response is obtained by dividing the increment of $R_{H_i}$ by $(1+n_i/x)^2$, up to a quadratic error. Summing over the parts yields
\[
 F(x)-F_0(x)=\frac{2p}{n^2}+\frac{S}{\al n^3}
 +O\left(\frac{A+1}{n^2}\right).
\]
Here the quadratic terms in the changes in the response are $O(n^{-2})$, and changing a denominator by $O(A+1)$ produces the displayed error. Setting $F(x)=1$, multiplying by $n$, and using
$S=p^2+p-2M_2(H)$ gives
\[
 w=f_r(p/n)-\frac{\al S_2}{n}
 -\frac{2M_2(H)}{\al n^2}
 +O\left(\frac{A+1}{n}\right),
\]
where the term $p/(\al n^2)$ is absorbed by the error. This is \eqref{eq:coarse}.

Assume now that $A=O(1)$. Let $D_*=\sum_i\frac{n_i}{(\mu+n_i)^2}$.
The root change is $O(p/n)$. Taylor expansion of the response equation, using \eqref{eq:R-expansion} and its differentiated form, gives
\[
 \lam(K+H)-\mu
 =\sum_i\frac{2p_i}{D_*(\mu+n_i)^2}
 +\frac1\mu\sum_i
 \frac{\sum_{v\in V_i}d_H(v)^2}{D_*(\mu+n_i)^2}
 +O\left(\frac{p^2+p}{n^3}\right).
\]
The base equation and the normalization condition imply
\[
 \mu=\al n-\frac{\al S_2}{n}+O(n^{-2}),
 \qquad
 \frac1{D_*(\mu+n_i)^2}
 =\frac1n-\frac{2a_i}{n^2}+O(n^{-3}).
\]
Putting these formulas together proves \eqref{eq:fine}; in the degree-square term the $a_i/n^2$ correction is $O((p^2+p)/n^3)$ and is absorbed.
\end{proof}

\begin{lemma}\label{lem:walk-defects}
There is an absolute constant $C$ such that every $p$-edge graph $J$, with $p\ge4$, satisfies
\begin{equation}\label{eq:walk-defects}
 |w_3(J)-2p^2|\le CM_2(J),\qquad
 |w_4(J)-(p^3+p^2)|\le CpM_2(J).
\end{equation}
\end{lemma}
\begin{proof}
Write $M=M_2(J)$. The assertion is clear for a star. The absolute bounds in Lemma~\ref{lem:walks} prove it when $M\ge p^2/16$, and also for $4\le p<16$ after increasing $C$. Suppose $p\ge16$ and $0<M<p^2/16$. Choose a vertex $v$ of degree $d=p-k$. The proof of Lemma~\ref{lem:walks} gives $pk\le4M$.

Delete the edges at $v$, leaving a $k$-edge graph $J_0$. For $z\ne v$, let $e_z$ be the indicator of $N_J(v)$ and put $b_z=d_{J_0}(z)$. If $B_0=\sum_{z\in N_J(v)}b_z\le2k$, then $\frac{w_3(J)}2=d(d+B_0)+ \sum_{xy\in E(J_0)}(e_x+b_x)(e_y+b_y)$.
The right side lies between $d^2$ and $p^2+2k^2+2k$, so
$|w_3(J)-2p^2|\le8pk=O(M)$.

Put $t_z=\sum_{y\in N_{J_0}(z)}(e_y+b_y)$. From $w_4=\|A(J)d_J\|^2$, $w_4(J)=(d+B_0)^2+d^3 +2d\sum_{z\in N_J(v)}t_z+\sum_z t_z^2$.
Here $\sum_zt_z\le k^2+3k$ and $\sum_z t_z^2\le2k\sum_z(e_z+b_z)^2 \le2k(p+k^2+5k)$.
Comparison with $p^3+p^2$ gives an error $O(p^2k)=O(pM)$, proving \eqref{eq:walk-defects}.
\end{proof}

\begin{lemma}\label{lem:relative-response}
Fix $B>0$. Uniformly for $p\ge4$, $u\ge p+1$, and $p\le Bx$, as $x\to\infty$,
\begin{equation}\label{eq:relative-response}
 R_{\rm st}(x;u,p)-R_J(x)
 =\frac{2M_2(J)}{x^3}\left(1+O_B(x^{-1/2})\right).
\end{equation}
\end{lemma}
\begin{proof}
Expand both responses in walks. Their terms of lengths zero and one agree, while the length-two difference is $2M_2(J)/x^3$. Lemma~\ref{lem:walk-defects} bounds the next two differences by
$CM_2(J)/x^4$ and $CpM_2(J)/x^5$. For either graph, the tail starting at length five is at most $\frac{Cp^{7/2}}{x^6(1-\sqrt{2p}/x)}$,
by the spectral theorem and $w_4\le2p(p^2+p)$. Outside the star class, Lemma~\ref{lem:walks} gives $M_2(J)\ge p-3\ge p/4$ after enlarging the constant for bounded $p$. Dividing the errors by $M_2(J)/x^3$ gives
\[
 O\left(x^{-1}+p/x^2+
 \frac{p^{5/2}}{x^3(1-\sqrt{2p}/x)}\right)=O_B(x^{-1/2}).
\]
For a star both sides vanish.
\end{proof}

For fixed $p$, the qualitative star--triangle alternative in complete multipartite extensions follows from \cite{FangTaitZhai} and is also recorded by Zhang \cite{ZhangWalks}; the linear-range version inside $\mathcal T_{n,r,p}$ is \cite[Theorem~1.4(ii)]{FLLM}. The next theorem supplies the uniform defect term needed when $p$ may grow linearly with $n$.

\begin{theorem}\label{thm:concentration}
Fix $r\ge3$ and $A,B>0$. Let $|n_i-n/r|\le A$, and let a $p$-edge graph $H$ lie in one part, where
$4\le p\le\min\{Bn,u-1\}$. If $G_{\rm st}$ replaces $H$ by a $p$-edge star in that part, then
\begin{equation}\label{eq:spectral-concentration}
 \lam(G_{\rm st})-\lam(K+H)
 =\frac{2M_2(H)}{\al n^2}\left(1+O_{r,A,B}(n^{-1/2})\right).
\end{equation}
The same conclusion holds when the internal edges lie in several parts of a common size $u$, provided that their total number $p$ satisfies $p\le u-1$ and they are replaced by one $p$-edge star in one of these parts. Here $M_2(H)$ is computed in the disjoint union of the internal graphs.
\end{theorem}
\begin{proof}
The two roots are $\al n+O_{r,A,B}(1)$ by \eqref{eq:fine}. On the interval between them, the derivatives of the responses are $-1/n+O(n^{-2})$. When the edges lie in one part, the exact change in that part summand is $\frac{R_{\rm st}(x)-R_H(x)} {(1+R_{\rm st}(x))(1+R_H(x))}$.
Each denominator factor is $1/\al+O(n^{-1})$. Equation~\eqref{eq:relative-response} and the mean-value theorem therefore give \eqref{eq:spectral-concentration}.

For several parts of a common size $u$, put $h=1+u/x$, $\delta_i=R_{H_i}(x)-u/x$, and $\delta=\sum_i\delta_i$. If $\delta_*$ is the increment for the single star, Lemma~\ref{lem:relative-response} applied to the disjoint union gives $\delta_*-\delta=\frac{2M_2(H)}{x^3} \left(1+O(n^{-1/2})\right)$.
The change in a full part response is $\psi_h(t)=t/[h(h+t)]$. Since $\delta_i=O(p_i/n^2)$ and $\psi_h''$ is bounded,
\[
 \left|\psi_h(\delta)-\sum_i\psi_h(\delta_i)\right|
 \le\frac{C}{n^4}\sum_{i<j}p_ip_j
 \le\frac{CM_2(H)}{n^4}.
\]
This error is smaller than the main term. The same mean-value argument completes the proof.
\end{proof}

\subsection{Two reduction lemmas}
\begin{lemma}\label{lem:preliminary}
Fix $r\ge3$, $B\ge1$ and $L\ge0$. Suppose $1\le q\le m$, $p\le Bn$, $z=o(n)$, $N_{K_{r+1}}(G)\le q c_r(n),\qquad \lam(G)\ge\al n+f_r(q/n)-L/n$.
With $A=\max_i|n_i-n/r|$, one has
\begin{equation}\label{eq:preliminary-short}
 S_2=O(z+1),\quad A=O(\sqrt{z+1}),\quad
 |p-q|=O(z+1),\quad M_2(H)=O(n(z+1)).
\end{equation}
If $q/m\to\gamma>0$, then $p-\Delta(H)=O(z+1)$.
\end{lemma}
\begin{proof}
Restore the missing cross-edges and write $\overline G=K+H$. Since $p\le Bn$, Weyl's inequality and the spectral lower bound give $\mu\ge\al n-O(\sqrt n)$. Equation \eqref{eq:variance-loss} first yields $S_2=O(n^{3/2})$, so $A=o(n)$ and all parts have size $(1/r+o(1))n$. On the interval between $\mu$ and $\lam(\overline G)$, equation \eqref{eq:R-expansion} changes the response by $O(n^{-1})$, while the base-response derivative has order $n^{-1}$. Hence $\lam(\overline G)-\mu=O(1)$, and another use of \eqref{eq:variance-loss} gives $S_2=O(n)$. Both roots are now $\al n+O(1)$, so \eqref{eq:coarse} applies.

By Lemma~\ref{lem:standard-completions},
$N_{K_{r+1}}(G)\ge \sum_i p_i\prod_{j\ne i}n_j-C_rzpn^{r-2}$.
Expanding each product as $m^{r-1}+O_r((A+1)n^{r-2})$ and using
$c_r(n)=m^{r-1}+O_r(n^{r-2})$ gives
\begin{equation}\label{eq:upper-p}
 p-q\le C\frac{p(A+z+1)+q}{n}
 \le C'(A+z+1).
\end{equation}

Drop the $M_2$ term in \eqref{eq:coarse} and compare with the spectral hypothesis. Since $f_r'$ is bounded on the relevant interval, $S_2\le C(A+z+1)$.
As $A^2\le S_2$, this gives $S_2=O(z+1)$ and $A=O(\sqrt{z+1})$. Equation \eqref{eq:upper-p} gives the upper bound for $p-q$; the same spectral comparison, using $f_r'(c)\ge2$, gives the lower bound. Keeping the $M_2$ term in \eqref{eq:coarse} yields $M_2(H)=O(n(z+1))$. If $q/m\to\gamma>0$, then $p=\Theta(n)$, and Lemma~\ref{lem:walks} implies
$p-\Delta(H)=O(z+1)$. These are the estimates in \eqref{eq:preliminary-short}.
\end{proof}

\begin{lemma}\label{lem:deletion}
In the setting of Lemma~\ref{lem:preliminary}, suppose $z=o(n)$ and let $\overline G=K+H$. If $p=o(n)$, then
\begin{equation}\label{eq:flat-deletion}
 \lam(\overline G)-\lam(G)=\frac{2z}{n}(1+o(1)).
\end{equation}
If $q/m\to\gamma>0$, let $v$ have internal degree $p-k$, let $z_h$ missing cross-edges meet $v$, and put $z_o=z-z_h$. With $\gamma_n=q/m$,
\begin{equation}\label{eq:hub-deletion}
 \lam(\overline G)-\lam(G)
 =\frac{(2+2\gamma_n/(r-1))z_h+2z_o}{n}+o(z/n).
\end{equation}
The error is zero when $z=0$.
\end{lemma}
\begin{proof}
Change the weights of all missing cross-edges continuously from zero to one. Every intermediate graph is connected for large $n$. Its largest eigenvalue lies between those of $G$ and $\overline G$, and both endpoints equal $\al n+O(1)$ by Lemma~\ref{lem:preliminary} and \eqref{eq:coarse}. Thus Lemma~\ref{lem:perron-coordinates} applies uniformly along the interpolation.

If $p=o(n)$, Lemma~\ref{lem:preliminary} gives $A=o(n)$, and $p+z=o(n)$. Hence \eqref{eq:flat-coordinates} holds throughout. The derivative of the simple Perron root with respect to the weight of an edge $ab$ is $2x_ax_b$. Integrating this derivative over the $z$ missing edges proves \eqref{eq:flat-deletion}.

Suppose that $q/m\to\gamma>0$. Lemma~\ref{lem:preliminary} gives
$A=o(n)$, $k=o(n)$ and $p=q+O(z+1)$. Formula \eqref{eq:hub-coordinates}, with $P_v=p-k$, gives uniformly
\[
 x_a=\frac{1+O((A+k+z+1)/n)}{\sqrt n}\quad(a\ne v),
 \qquad
 x_v=\frac{1+\gamma_n/(r-1)+O((A+k+z+1)/n)}{\sqrt n}.
\]
A missing edge meeting $v$ therefore contributes
$[2+2\gamma_n/(r-1)]/n+o(n^{-1})$ to the integrated change, while a missing edge avoiding $v$ contributes $2/n+o(n^{-1})$. Since $A+k+z=o(n)$, the sum of the errors is $o(z/n)$. This proves \eqref{eq:hub-deletion}. If $z=0$, no interpolation is made and the error is zero.
\end{proof}

\begin{proposition}\label{thm:reduction}
Fix $r\ge3$, $B\ge1$, $0<\tau<1$ and $L\ge0$. There are $\eps_1,C,n_0>0$ such that, if $n\ge n_0$, $1\le q\le(1-\tau)m$, $p\le Bn$, $z\le\eps_1n$, $N_{K_{r+1}}(G)\le q c_r(n),\qquad \lam(G)\ge\al n+f_r(q/n)-L/n$,
then
\begin{equation}\label{eq:bounded-conclusion}
 z+|p-q|+\max_i|n_i-n/r|\le C,
 \qquad M_2(H)\le Cn.
\end{equation}
If $q\ge\eta m$ for fixed $\eta>0$, then $p-\Delta(H)=O(1)$.
\end{proposition}
\begin{proof}
It is enough to exclude a sequence with $z=o(n)$ and $z\to\infty$. Pass to a subsequence with $q/m\to\gamma\in[0,1-\tau]$. By Lemma~\ref{lem:preliminary}, $A=o(z),\qquad S_2=O(z),\qquad \ell:=p-q=O(z)$.
If $\gamma=0$, then $p=o(n)$. The sharper first inequality in \eqref{eq:upper-p} gives $\ell_+=o(z)$, where $\ell_+=\max\{\ell,0\}$. Equations \eqref{eq:coarse} and \eqref{eq:flat-deletion} give
\[
 n\{\lam(G)-\al n-f_r(q/n)\}
 \le(2+o(1))\ell_+-\al S_2-2z+o(z)
 \le-2z+o(z),
\]
contrary to the fixed lower bound $-L$.

Suppose $\gamma>0$. Let $v$ have internal degree $p-k$, where $k=O(z)$, and split the missing edges into $z_h$ meeting $v$ and $z_o$ avoiding it. Put $\gamma_n=q/m,\qquad b_n=\frac{2\gamma_n}{r-1},\qquad a_n=2+b_n$.
Each remaining internal edge is disjoint from all but at most two of the edges at $v$, and pairs of remaining edges contribute only $O(k^2)$. Hence $M_2(H)=pk+O(k^2+k)$.
Expanding $f_r(p/n)$ about $q/n$ in \eqref{eq:coarse}, and then applying \eqref{eq:hub-deletion}, yields
\begin{equation}\label{eq:scaled-spectrum}
 n\{\lam(G)-\al n-f_r(q/n)\}
 =a_n\ell-\al S_2-b_nk-a_nz_h-2z_o+o(z).
\end{equation}
Here $\ell^2/n$, $k^2/n$, and $A+1$ are $o(z)$.

We refine the clique count according to the position of a missing edge. Lemma~\ref{lem:standard-completions} shows that a missing cross-edge meeting $v$ destroys at most
$pm^{r-2}+O_r(p(A+1)n^{r-3}+n^{r-2})$ standard completions, while a missing edge avoiding $v$ destroys only $O_r((k+1)n^{r-2})$ completions. Summing these estimates gives
\[
 \ell\le(p/m)z_h+O_r\left(A+1+\frac{z(A+k+1)}n\right)
 =\gamma_nz_h+o(z).
\]
Putting this in \eqref{eq:scaled-spectrum} gives
\begin{equation}\label{eq:weighted-defect}
 n\{\lam(G)-\al n-f_r(q/n)\}
 \le-\al S_2-b_nk-2z_o-a_n(1-\gamma_n)z_h+o(z).
\end{equation}
Since $1-\gamma_n\ge\tau$ and $a_n\ge2$, the right side of \eqref{eq:weighted-defect} is at most $-2\tau z+o(z)$, another contradiction.

If no fixed $\eps_1,C$ forced $z\le C$, then for each integer $j$ one could choose a counterexample with $n\ge j$, $z\le n/j$, and $z>j$. This would produce a forbidden sequence with $z=o(n)$ and $z\to\infty$. Thus a fixed small linear range forces $z=O(1)$. Lemma~\ref{lem:preliminary} gives \eqref{eq:bounded-conclusion}; when $q\ge\eta m$, Lemma~\ref{lem:walks} also gives $p-\Delta(H)=O(1)$.
\end{proof}

Near $q=m$ the coefficient of $z_h$ in the preceding inequality vanishes. The correct variable is the actual degree defect of the vertex meeting almost all internal edges.

\begin{lemma}\label{lem:core-expansion}
Suppose $q/m\to1$, $z=o(n)$ and the hypotheses of Corollary~\ref{cor:global-entry} hold. Choose a maximum-internal-degree vertex $v$. In the partition induced on $G-v$, let $b_i$ be the part sizes, let $k$ and $z_o$ be the numbers of internal and missing cross-edges, and put
\[
 N=n-1,
 \quad b_0=N/r,
 \quad S_c=\sum_i(b_i-b_0)^2,
 \quad A_c=\max_i|b_i-b_0|,
 \quad t=n-1-d_G(v),
 \quad h=m-q,
 \quad \Delta=t-h.
\]
If $w_i$ is the number of non-neighbors of $v$ in core part $i$, set
\[
 F_0(x)=\sum_i\frac{b_i}{x+b_i},\qquad
 W(x)=\sum_i\frac{w_i}{x+b_i},\qquad
 Q(x)=\sum_i\frac{w_i^2}{x+b_i}.
\]
For the graph $C_v$ obtained before the core edits, its spectral radius is the root near $\al n$ of
\begin{equation}\label{eq:core-response}
 F_0(x)+\frac{x(1-W(x))^2}{x(x+1)-t+Q(x)}=1.
\end{equation}
Let $\rho_n(t)$ be the root near $\al n$ of the comparison equation
\begin{equation}\label{eq:proxy}
 \frac{rb_0}{x+b_0}+
 \frac{x\left(1-t/(x+b_0)\right)^2}{x(x+1)-t}=1.
\end{equation}
Then
\begin{equation}\label{eq:core-spectrum}
 \begin{split}
 n\{\lam(G)-\rho_n(h)\}
 ={}&-\al S_c+2k-2z_o-\frac2\al\Delta\\
 &+o(S_c+k+z_o+|\Delta|+1),
 \end{split}
\end{equation}
and $\rho_n(h)=\al n+f_r(q/n)+O_r(n^{-1})$. If $\chi_s=1$ for $s=0$ and $0$ otherwise, the clique bound implies
\begin{equation}\label{eq:core-budget}
 \Delta\ge k-\chi_s-o(S_c+k+z_o+1).
\end{equation}
If $S_c,k,z_o$ are bounded and $s>0$, then $\Delta\ge k$ for large $n$.
\end{lemma}
\begin{proof}
Theorem~\ref{thm:profiles} and Lemma~\ref{lem:preliminary} show that a maximum-internal-degree vertex is a vertex meeting almost all internal edges and give
\begin{equation}\label{eq:core-preliminary}
 S_c=o(n),\qquad k=o(n),\qquad z_o=o(n),\qquad t=o(n).
\end{equation}
Indeed, if $v$ belongs to the original part $V_i$ and $z_h$ missing cross-edges meet $v$, then $t=n_i-1-(p-k)+z_h=n_i-1-p+k+z_h=h+o(n)$,
because $n_i=m+o(n)$, $p=q+o(n)$, and $k+z_h=o(n)$. Throughout the rest of the proof we work under \eqref{eq:core-preliminary}.

Equation~\eqref{eq:core-response} is the case $|S|=1$ of the one-set walk-series identity in Section~1 of Zhang \cite{ZhangWalks}. For completeness, we record only its evaluation on the present complete multipartite core. Write $d_i=b_i-w_i$ and $d=N-t$. The resolvent term is
\[
 \frac1x\left(d-\sum_i\frac{d_i^2}{x+b_i}\right)
 +\frac{\left(\sum_i d_i/(x+b_i)\right)^2}{1-F_0(x)}.
\]
Using $\sum_i d_i/(x+b_i)=F_0-W$ and
$\sum_i d_i^2/(x+b_i)=N-xF_0-2t+2xW+Q$ gives \eqref{eq:core-response}.

The cone over $C$ has its largest root at $\al n+O(1)$ because $S_c=o(n)$. During the deletion of the $t=o(n)$ edges incident with $v$, Lemma~\ref{lem:perron-coordinates} gives $O(n^{-1/2})$ for every unit Perron coordinate. Integration of the edge derivatives therefore shows that the root changes by $o(1)$. On this root interval, $W=O(t/n)$ and $Q=O(t^2/n)$, so the derivative of the extra summand in \eqref{eq:core-response} is $O(n^{-2})$, whereas
$F_0'(x)=-1/n+O(n^{-2})$. Thus the relevant root is unique and the response derivative is $-1/n+O(n^{-2})$.

We compare it with the following simpler equation. Since $\sum_i(b_i-b_0)=0$,
\[
 F_0(x)-\frac{rb_0}{x+b_0}
 =-\frac{xS_c}{(x+b_0)^3}
 +O_r(A_cS_c/n^3)
 =-\frac{\al S_c}{n^2}
 +O_r((A_c+1)S_c/n^3).
\]
Moreover,
\[
 W(x)-\frac{t}{x+b_0}=O_r(A_ct/n^2),
 \qquad
 0\le Q(x)\le O_r(t^2/n).
\]
The partial derivatives of the extra summand in \eqref{eq:core-response} with respect to $W$ and $Q$ have orders $O(n^{-1})$ and $O(n^{-3})$, respectively. By the mean-value theorem and the response derivative,
\begin{equation}\label{eq:core-vs-proxy}
 n\{\lam(C_v)-\rho_n(t)\}
 =-\al S_c+O_r\left(
 \frac{(A_c+1)S_c}{n}+\frac{A_ct}{n}+\frac{t^2}{n^2}
 \right).
\end{equation}
Differentiating \eqref{eq:proxy} with respect to $t$ gives
\begin{equation}\label{eq:proxy-derivative}
 \rho_n'(t)=-\frac{2}{\al n}
 \left(1+O_r\left(\frac{t+1}{n}\right)\right),
\end{equation}
because the partial derivatives in $t$ and $x$ are
$-2/(\al n^2)(1+O_r((t+1)/n))$ and $-1/n+O_r(n^{-2})$. Hence \eqref{eq:proxy-derivative} gives $n\{\rho_n(t)-\rho_n(h)\} =-\frac2\al\Delta+o(|\Delta|)$,
with zero error when $t=h$.

Change the weights of all core edits continuously from $C_v$ to $G$. The complete-cross graph containing every internal edge that occurs during this interpolation has spectral radius $\al n+O(1)$ by \eqref{eq:coarse}, and every intermediate graph has spectral radius $\al n+O(1)$ by monotonicity and the average-degree lower bound. In the full $r$-part partition, $v$ meets all but $k=o(n)$ internal edges, while the part deviations and the number of lowered cross-edge weights are $o(n)$. Lemma~\ref{lem:perron-coordinates} therefore shows that every core coordinate is $(1+o(1))/\sqrt n$, uniformly during the interpolation. Integrating the edge derivatives gives
$n\{\lam(G)-\lam(C_v)\}=2k-2z_o+o(k+z_o)$.
In \eqref{eq:core-vs-proxy}, the error is $o(S_c+1)$ because $A_c^2\le S_c$ and $S_c,t=o(n)$. Combining the last three displays proves \eqref{eq:core-spectrum}.

To estimate $\rho_n(h)$, take a complete $r$-partite graph whose part sizes differ from $n/r$ by $O(1)$ and whose chosen part has size $m+1$, and insert a $q$-edge star there. For $s=0$, offset one other part to $m-1$; for $s>0$, distribute the remaining $s-1$ extra vertices among the other parts. Deleting the star center leaves a core whose part sizes differ from $(n-1)/r$ by $O(1)$ and exactly $h=m-q$ non-neighbors. Equations \eqref{eq:core-vs-proxy} and \eqref{eq:fine}, with $M_2=0$, give
$\rho_n(h)=\al n+f_r(q/n)+O_r(n^{-1})$.

It remains to prove the clique-count inequality. Write $b_i=m+c_i$, where $\sum_i c_i=s-1$. The complete core has $\prod_i b_i=m^r+(s-1)m^{r-1} +O_r((S_c+1)n^{r-2})$ cliques with one vertex in each part. A non-neighbor of $v$ destroys at most
$m^{r-1}+O_r((A_c+1)n^{r-2})$ of them, and a missing core cross-edge destroys at most $O_r(n^{r-2})$. Thus the cliques containing $v$ number at least $m^r+(s-1-t)m^{r-1} -O_r((S_c+1+t(A_c+1)+z_o)n^{r-2})$.
The $k$ internal core edges have disjoint standard-completion families, all avoiding $v$, of total size at least $km^{r-1}-O_r((k(A_c+1)+kz_o)n^{r-2})$.
On the other hand, $qc_r(n)=m^r+(\max\{s-1,0\}-h)m^{r-1} +O_r((h+1)n^{r-2})$.
Comparison with the clique hypothesis gives
\[
 \Delta\ge k-\chi_s-
 O_r\left(\frac{S_c+1+(t+k)(A_c+1)+z_o+kz_o+h}{n}\right).
\]
The error is $o(S_c+k+z_o+1)$ because $A_c\le\sqrt{S_c}$ and all the core quantities are $o(n)$. This proves \eqref{eq:core-budget}. If $S_c,k,z_o$ are bounded and $s>0$, the error tends to zero and integrality yields $\Delta\ge k$.
\end{proof}

\begin{theorem}\label{thm:core}
Fix $r\ge3$, $B\ge1$ and $L\ge0$. There are $\delta,\eps_0,C>0$ such that, for all sufficiently large $n$, if
\[
 (1-\delta)m\le q\le m,
 \quad p\le Bn,
 \quad z\le\eps_0n,
 \quad N_{K_{r+1}}(G)\le q c_r(n),
 \quad \lam(G)\ge\al n+f_r(q/n)-L/n,
\]
then some vertex $v$ satisfies, in the partition induced on $G-v$,
\begin{equation}\label{eq:core-conclusion}
 \max_i\left|b_i-\frac{n-1}{r}\right|+k+z_o+
 \left|n-1-d_G(v)-(m-q)\right|\le C.
\end{equation}
\end{theorem}
\begin{proof}
It is enough to treat a sequence with $q/m\to1$ and $z=o(n)$. Theorem~\ref{thm:profiles} and Lemma~\ref{lem:preliminary} give
$S_c,k,z_o,t=o(n)$, so Lemma~\ref{lem:core-expansion} applies. Put
$u=\Delta-k$ and $T=S_c+k+z_o+1$. Equations \eqref{eq:core-spectrum} and \eqref{eq:core-budget} give
\[
 n\{\lam(G)-\rho_n(h)\}
 =-\al S_c-\frac{2}{r-1}k-2z_o-\frac2\al u
 +o(T+|u|),
\]
while $u_-\le1+o(T)$. The spectral lower bound therefore forces
$S_c,k,z_o,u_+$ to be bounded, and then all quantities in \eqref{eq:core-conclusion} are bounded. Suppose that no such fixed constants exist. Taking one counterexample for each $j$ gives a sequence with $q/m\to1$ and $z=o(n)$, contrary to what we have just proved. Hence fixed constants $\delta,\eps_0,C$ exist.
\end{proof}

\section{Proofs of the main theorems}\label{sec:exact}
\subsection{The sublinear case}\label{sec:sublinear}
Throughout this subsection $r$ is fixed, $\alpha=(r-1)/r$, and $q=o(n)$. All assertions about such sequences are uniform in the sense that their failure for arbitrarily large orders would produce a contradicting sequence. In particular, bounded values of $q$ are included.

\begin{lemma}\label{lem:second-layer}
Let $G$ be an extremal graph for the strict problem, and suppose $q=o(n)$. For sufficiently large $n$, there is a maximum $r$-cut for which $p=q$, $z=0$, and the complete multipartite base is $B^-_{n,r}$ or $B^+_{n,r}$.
\end{lemma}
\begin{proof}
For every value of $s$, there is a $q$-edge star extension $W$ on an almost balanced base that satisfies the strict inequality. Use the graph $E_{n,r,q}$ defined in Section~\ref{sec:results}, and use $U_{m,q}$ when $r=3$ and $s=2$. Its base has squared part-size defect two. Formula \eqref{eq:fine} gives $\lam(W)=\alpha n+f_r(q/n)-O_r(n^{-1})$.
Thus Corollary~\ref{cor:global-entry} applies to an extremal graph. Proposition~\ref{thm:reduction} then gives bounded part-size differences, $z=O(1)$, and $p=q+O(1)$. Its standard completions satisfy $N_{K_{r+1}}(G)\ge p c_r(n)-O_r(pn^{r-2})$.
Since $p=o(n)$, the integer inequality $p\ge q+1$ is impossible. Hence $p\le q$.
Let $D$ be the difference between the squared part-size sum of the base and that of $T_{n,r}$. It is a nonnegative even integer. Comparison with $W$ by \eqref{eq:fine}, using $\sum_vd_H(v)^2\le p^2+p\le q^2+q$, gives
\begin{equation}\label{eq:strict-coarse}
 \lam(W)-\lam(G)\ge
 \frac{\alpha(D-2)+2(q-p+z)}n
 -O_r\left(\frac{q+1}{n^2}+\frac{q^2+q}{n^3}\right).
\end{equation}
Lemma~\ref{lem:perron-coordinates}, applied while the missing-edge weights are restored, gives a deletion cost of
$2z/n+O_r((p+1)n^{-2})$. This justifies the deletion term in \eqref{eq:strict-coarse}; since $p=o(n)$, its error is $o(n^{-1})$.
If $p\le q-1$ or $z\ge1$, the leading coefficient is at least $2-2\alpha>0$, contradicting maximality. Thus $p=q$ and $z=0$. A balanced base would already produce $q c_r(n)$ standard cliques, and so $D\ne0$. Values $D\ge4$ also contradict \eqref{eq:strict-coarse}. Therefore $D=2$.

Balancing two integer parts $a\ge b+2$ reduces their squared sum by $2(a-b-1)\ge2$. Repeated balancing reaches the Tur\'an vector. Total defect two therefore permits just one move, with $a-b=2$. Reversing it transfers a vertex between two equal Tur\'an parts. These are exactly the bases $B^-$ and $B^+$.
\end{proof}

We will use a base comparison without an error term. Put $g_x(t)=t/(x+t)$, and let $F_0^+$ and $F_0^-$ denote the base responses of $B^+_{n,r}$ and $B^-_{n,r}$, respectively. Whenever both almost balanced bases exist,
\begin{equation}\label{eq:third-difference}
 \begin{aligned}
 F^+_0(x)-F^-_0(x)
 &=g_x(m+2)-3g_x(m+1)+3g_x(m)-g_x(m-1)\\
 &=\frac{6x}{(x+m-1)(x+m)(x+m+1)(x+m+2)}>0.
 \end{aligned}
\end{equation}
The identity follows after putting the four fractions over a common denominator. At either base root, $x=\alpha n+O(n^{-1})$ and the base derivative is $-1/n+O(n^{-3})$. Hence,
\begin{equation}\label{eq:base-third-gap}
 \lam(B^+)-\lam(B^-)=\frac{6\alpha}{n^2}+O_r(n^{-3}).
\end{equation}
If the same internal graph is inserted in a part of the same size on each base, the changes in the response are equal. Then \eqref{eq:third-difference} also gives a strict, exact comparison of the two joined graphs.

\subsection*{The first almost balanced base}
Assume first that $s=0$. In $B^-$ let $p_+,p_-,p_0$ count the internal edges in parts of sizes $m+1,m-1,m$, respectively. The corresponding standard-completion counts divided by $c_r(n)$ are $1-1/m$, $1+1/m$, and $1-1/m^2$. Thus the strict inequality implies $m(p_--p_+)<p_0\le q=o(m)$,
and integrality gives $p_-\le p_+$. The term involving the part sizes in \eqref{eq:fine} satisfies $\sum_i a_i p_i=p_+-p_-\ge0$. A star in an $m$-part attains zero and has squared-degree sum $q^2+q$.

For $q\ge4$, every internal graph other than one star loses at least $2(q-3)$ in its squared-degree sum, by Lemma~\ref{lem:walks}. This loss dominates the error $O_r((q^2+q)n^{-3})$ in \eqref{eq:fine}. The same conclusion holds for $q=2$, because two disjoint edges have smaller squared-degree sum. For $q=3$, the only additional graph with squared-degree sum twelve is a triangle. A triangle in an $m$-part does not satisfy the strict inequality: standard completions leave a deficit only $O_r(n^{r-3})$, whereas the triangle with one vertex from each of $r-2$ other parts creates $\Theta_r(n^{r-2})$ additional cliques. A triangle or a star entirely in the $(m+1)$-part has a smaller spectral radius because of the part-size term; a star in the $(m-1)$-part does not satisfy the strict inequality. For $q=1$, the part sizes alone give the assertion. Thus the prescribed star in an $m$-part is the unique extremal graph on $B^-$.

If $s>0$, let $p_L,p_M,p_D$ count edges in the parts of sizes $m+1,m,m-1$. The standard-completion ratios give the necessary inequality
\[
 -\frac{p_L}{m^2}
 +p_M\left(\frac1m-\frac1{m^2}-\frac1{m^3}\right)
 +p_D\left(\frac2m+\frac1{m^2}\right)<0.
\]
Since $q=o(m)$, the presence of even one edge in an $m$- or $(m-1)$-part makes the left side positive for large $m$. Hence all edges lie in $(m+1)$-parts. Their weights coincide, so the same squared-degree comparison gives one star. Again a triangle in one such part does not satisfy the strict inequality: its additional clique count has order $n^{r-2}$, larger than the available standard-completion deficit of order $n^{r-3}$. This also proves the cases $q=1,2,3$.

\subsection*{The second almost balanced base}
Write $p_A,p_L,p_M$ for internal edges in parts of sizes $m+2,m+1,m$. The completion ratios are $1-1/(m+1)$, $1-1/(m+1)^2$, and $1+1/(m+1)$. Hence $(m+1)(p_M-p_A)<p_L\le q=o(m)$,
so $p_M\le p_A$. Relative to putting every edge in an $(m+1)$-part, the extra part-size term is $p_A-p_M\ge0$. The star in an $(m+1)$-part therefore maximizes both terms in \eqref{eq:fine}. The cases $q\le3$ are handled as above: a triangle in that part violates the count bound, a star or triangle in the $(m+2)$-part has a strictly smaller part-size term, and a star in an $m$-part does not satisfy the strict inequality. Thus the prescribed star is uniquely optimal on $B^+$.
If $B^-$ also exists, its optimal star has the same part size. Equation \eqref{eq:third-difference} makes $B^+$ strictly better, for every $q$ under discussion.

\subsection*{The remaining part-size case}
There is no $(m+1)$-part on $B^+$ in this case. Let $k$ edges lie in its $m$-parts. The strict count inequality gives $k<q/2$. The other $q-k$ edges lie in the $(m+2)$-part, and $\sum_v d_H(v)^2\le(q-k)^2+(q-k)+k^2+k=q^2+q-2k(q-k)$.
Suppose $r\ge4$, so $B^-$ also exists. Its optimal star is in an $(m+1)$-part. By \eqref{eq:fine} and \eqref{eq:base-third-gap},
\begin{equation}\label{eq:s2-allocation}
 \begin{split}
 n^2\bigl(\lam(B^++H)-\lam(B^-+K_{1,q})\bigr)
 \le 6\alpha-4q+8k-\frac{2k(q-k)}\alpha
 +O_r\left(\frac{q^2+q+1}{n}\right).
 \end{split}
\end{equation}
For $q\ge2$ the main expression is at most $-(q-1)$. Indeed it is at most $6-4q+8k-2k(q-k)$, a convex quadratic in $k$. For $q=2$ only $k=0$ is possible. For $q\ge3$, its values at $k=0$ and $k=(q-1)/2$ are $6-4q$ and $(5-q^2)/2$, both at most $-(q-1)$. Since $q=o(n)$, the error in \eqref{eq:s2-allocation} is smaller than this loss. The star on $B^-$ is therefore uniquely optimal for every $q\ge2$.

For $q=1$ only one edge in the $(m+2)$-part of $B^+$ satisfies the bound. Comparison with the edge in an $(m+1)$-part of $B^-$ is exact to second order:
\[
 \lam(B^++K_2)-\lam(B^-+K_2)
 =\frac{6\alpha-4}{n^2}+O_r(n^{-3})
 =\frac{2(r-3)}{rn^2}+O_r(n^{-3})>0.
\]
This proves the different one-edge choice and \eqref{eq:one-edge-gap}.

When $r=3$, only $B^+=K_3(m+2,m,m)$ exists. If $q=1,2$, the inequality $k<q/2$ forces $k=0$ and the best internal graph is the star. If $q=3$, either $k=0$ or $k=1$. For $k=0$, the squared-degree sum is at most twelve, with equality for a star or a triangle. For $k=1$, it is at most eight, with equality only for a $P_3$ in the largest part and one edge in a smaller part. Relative to the three-edge star, this allocation improves the part-size term by $8/n^2$ and loses $4/(\alpha n^2)=6/n^2$ in the squared-degree term. Hence $\lam(M_m)-\lam(U_{m,3})=\frac2{n^2}+O(n^{-3})$.
A triangle in the largest part has the same first two terms in \eqref{eq:fine} as $U_{m,3}$, proving the second comparison in \eqref{eq:mixed-gap}. Every other allocation with $k=1$ has strictly smaller squared-degree sum than $M_m$. Together with \eqref{eq:exception-counts}, this proves that $M_m$ is the unique extremal graph at $q=3$.

Finally let $r=3$, $s=2$, and $q\ge4$. If $k\ge1$, comparison with the star gives
\[
 n^2\bigl(\lam(U_{m,q})-\lam(B^++H)\bigr)
 \ge k\bigl(3(q-k)-8\bigr)-O((q^2+q)/n).
\]
For $4\le q\le15$, the integer condition $k<q/2$ gives $q-k\ge3$, making the main term positive. For $q\ge16$ it is at least $kq$, which dominates the error when $q=o(n)$. If $k=0$ but $H$ is not a star, its squared-degree loss of at least $2(q-3)$ again dominates the error. Thus $U$ is the unique extremal graph in this range for $q\ge4$.

We have proved Theorem~\ref{thm:main} when $q/m\to0$, including every fixed value of $q$.

\subsection{The linear case and the endpoint}\label{sec:terminal}
Put $\alpha=(r-1)/r$ and $\kappa=0$ if $s=0$, and $\kappa=1$ otherwise. We compare every graph with $T^*_{n,r,q}$. This graph has exactly $q c_r(n)$ copies, so it does not satisfy the strict inequality.

\subsection*{The interior range}
Consider a sequence with $q/m\to\gamma\in(0,1)$ and spectral radius at least that of the stated graph. The stated graph has spectral radius $\alpha n+f_r(q/n)+O_r(n^{-1})$. Corollary~\ref{cor:global-entry} and Proposition~\ref{thm:reduction} give bounded part deviations, bounded missing edges, and a vertex $v$ incident with all but $k=O(1)$ internal edges. Relabel the parts so that $v$ lies in the first part, and write
\[
 n_i=m+t_i,\quad \sum_i t_i=s,\quad t=t_1,\quad
 p=q+\ell,\quad D_2=\sum_i t_i^2-s,\quad d=t-\kappa.
\]
Let $z_h$ missing edges meet $v$, and let $z_o$ missing edges avoid it. Put $e=\ell-z_h$. After passing to a subsequence, we may assume that all these integers are fixed. Define
\[
 a_\gamma=2+\frac{2\gamma}{r-1},\qquad
 b_\gamma=\frac{2\gamma}{r-1},\qquad
 w_\gamma=\frac{4\gamma}{r}+\frac{2\gamma^2}{r(r-1)}.
\]

\begin{lemma}\label{lem:integer}
In this notation,
\begin{align}
 e&\le\gamma d-(1-\gamma)z_h,\label{eq:integer-budget}\\
 n\bigl(\lam(G)-\lam(T^*_{n,r,q})\bigr)
 &=-\alpha D_2+a_\gamma e-w_\gamma d-b_\gamma k-2z_o+o(1).
 \label{eq:integer}
\end{align}
\end{lemma}
\begin{proof}
The $p-k$ internal edges at $v$ give their standard completions by choosing a neighbor of $v$ in every other part. The remaining $k$ internal edges give $k m^{r-1}+O_r(n^{r-2})$ further standard cliques. Missing edges away from $v$ destroy only $O_r(n^{r-2})$ members of these families. Since $c_r(n)=m^{r-1}+(s-\kappa)m^{r-2}+O_r(n^{r-3})$,
we obtain
\[
 N_{K_{r+1}}(G)-q c_r(n)
 \ge\bigl(\ell+(q/m)(\kappa-t-z_h)\bigr)m^{r-1}-O_r(n^{r-2}).
\]
All uncounted cliques have nonnegative contribution. Division by $m^{r-1}$ and passage to the limit prove \eqref{eq:integer-budget}.

A $p$-edge star fits in the first part because $\gamma<1$. Move the remaining $k$ edges to $v$, and then add back the missing edges. All part deviations and all edit counts other than the star are bounded. Lemma~\ref{lem:perron-coordinates} therefore gives, uniformly during these changes,
\[
 x_u=\frac{1+O_r(n^{-1})}{\sqrt n}\ (u\ne v),\qquad
 x_v=\frac{1+\gamma_n/(r-1)+O_r(n^{-1})}{\sqrt n},\quad \gamma_n=q/m.
\]
Along a change of edge weights the simple Perron root satisfies $d\lam/dt=2x_a(t)x_b(t)$ for an edge $ab$. Replacing one edge avoiding $v$ by one meeting $v$ increases the root by $b_\gamma/n+o(n^{-1})$. Restoring a missing edge meeting $v$ increases it by $a_\gamma/n+o(n^{-1})$, and restoring a missing edge avoiding $v$ increases it by $2/n+o(n^{-1})$. Summing over the bounded changes gives the two correction terms used below.

We next compare the stars directly. Write the star change in the response as
\[
 d_*(x,p)=\frac{p(2x+p+1)}{x(x^2-p)},\qquad
 \Psi(x,u,p)=\frac{d_*(x,p)}{(1+u/x)(1+u/x+d_*(x,p))}.
\]
At $x=\alpha n+O(1)$, $u=n/r+O(1)$ and $p=q+O(1)$,
\[
 \Psi_p=\frac{f_r'(q/n)}{n^2}+O_r(n^{-3}),\qquad
 \Psi_u=-\frac{2f_r(q/n)}{n^2}+O_r(n^{-3}).
\]
The base-response difference from $T_{n,r}$ is $-\alpha D_2/n^2+O_r(n^{-3})$. Integrating over the bounded changes of $p,u$, and using response derivative $-1/n+O_r(n^{-2})$, gives the star comparison
\[
 n(\lam_{\rm st}-\lam(T^*_{n,r,q}))
 =-\alpha D_2+a_\gamma\ell-w_\gamma d+o(1).
\]
Subtracting these two changes and using $e=\ell-z_h$ proves \eqref{eq:integer}.
\end{proof}

\begin{lemma}\label{lem:integer-defect}
The integer $D_2$ is nonnegative and even, with equality exactly at a Tur\'an vector. For $d>0$ one has $D_2\ge d(d+1)$. If $d=-u<0$, then $D_2\ge u(u+1)$ for $s=0$, and $D_2\ge u(u-1)$ for $s>0$.
\end{lemma}
\begin{proof}
The identity $D_2=\sum_i t_i(t_i-1)$ proves the first assertion. If $s>0$, the term belonging to $t=d+1$ gives $d(d+1)$ for positive $d$, and the term belonging to $t=1-u$ gives $u(u-1)$ for negative $d$. If $s=0$, the other offsets sum to $-t$, so their squares sum to at least $|t|$. Hence $D_2\ge t^2+|t|$.
\end{proof}

First assume that $(r,s)\ne(3,2)$. Since $q\to\infty$, the exception with one edge does not occur. The stated graph is a $q$-edge star in a part of size $m+\kappa$ on an almost balanced base. For this graph, the right side of \eqref{eq:integer} is $-2\alpha+o(1)$. Hence every graph with at least the same spectral radius must have limiting value at least $-2\alpha$.
Set $c_\gamma=a_\gamma\gamma-w_\gamma=2\gamma(r-2+\gamma)/r$. Then $0<c_\gamma<2\alpha$.
If $d\ge2$, \eqref{eq:integer-budget} and Lemma~\ref{lem:integer-defect} bound the right side above by
$-\alpha d(d+1)+c_\gamma d<-\alpha d(d-1)\le-2\alpha$.
If $d=1$, integrality in \eqref{eq:integer-budget} gives $e\le0$, and the right side is at most $-2\alpha-w_\gamma<-2\alpha$.
If $s=0$ and $d\le-1$, or if $s>0$ and $d\le-2$, the negative-$d$ bounds give the same strict conclusion.
The remaining negative possibility is $s>0$, $d=-1$. Here $e\le-1$, so the right side is at most $-a_\gamma+w_\gamma$. This is strictly less than $-2\alpha$, since
\[
 a_\gamma-w_\gamma-2\alpha
 =\frac{2(1-\gamma)(r-1+\gamma)}{r(r-1)}>0.
\]
It follows that $d=0$.

If $z_h\ge1$, \eqref{eq:integer-budget} forces $e\le-1$; this, or $e\le-1$ directly, gives a value at most $-a_\gamma<-2\alpha$. Thus $e=z_h=0$. A missing edge away from $v$ contributes $-2$ to the right side and is also excluded. Hence $z=0$ and $p=q$. A balanced base does not satisfy the strict inequality, and a defect at least four gives a value below $-2\alpha$. Therefore $D_2=2$, and the equality then forces $k=0$.
We have a single star on an almost balanced base, in a part of size $m+\kappa$. The definition of the bases leaves exactly the graphs stated in the theorem. When both are possible, \eqref{eq:third-difference} shows that $B^+$ has larger spectral radius. This proves the result when $0<\gamma<1$ and $(r,s)\ne(3,2)$.

\subsection*{A tripartite exception}
Let $r=3$, $s=2$, and $q/m\to\gamma\in(0,1)$. Write the three offsets as $t,a,b$, where $t+a+b=2$, and put $A_\gamma=(4\gamma+\gamma^2)/3$ and $U_\gamma=(2-5\gamma-2\gamma^2)/3$.
Subtracting the expression for $V$ in \eqref{eq:integer} gives
\[
 n(\lam(G)-\lam(V))=P_\gamma+o(1),\qquad
 P_\gamma=-\frac23D_2+(2+\gamma)(e+1)-A_\gamma t-\gamma k-2z_o.
\]
The count inequality gives $e\le\gamma(t-1)-(1-\gamma)z_h$. Also
$n(\lam(U)-\lam(V))=U_\gamma+o(1)$.
Thus every possible extremal graph must satisfy $P_\gamma\ge\max\{0,U_\gamma\}$.

If $t\le-1$, then $e\le-1$ and $D_2\ge3t^2/2-2t$. Hence
$P_\gamma\le-t^2+(4/3-A_\gamma)t\le-2/3$.
If $t=0$, then $e\le-1$ and $D_2\ge0$; nonnegativity requires $a=b=1$, $e=-1$, and $k=z_o=0$. Deleting $v$ leaves $K_3(m-1,m+1,m+1)$ and $v$ has $m-q$ non-neighbors there. Lemma~\ref{lem:low-core}, proved below and valid also away from the endpoint, forces these non-neighbors into the $(m-1)$-part. The graph is $V$.

If $t=1$ and $\{a,b\}\ne\{0,1\}$, then $D_2\ge4$ and $e\le0$, so
$P_\gamma\le(-2-\gamma-\gamma^2)/3<0$.
For $\{a,b\}=\{0,1\}$, the base is balanced. Values $e\le-1$ give a negative right side. If $e=0$, \eqref{eq:integer-budget} and $\gamma<1$ force $z_h=0$. With $z_o=0$ the $q$ standard-completion families violate the strict inequality; with $z_o\ge1$ the right side is at most $\gamma-A_\gamma<0$.

If $t=2$, write $b=-a$. Then $e\le0$ and $P_\gamma\le U_\gamma-\frac{4a^2}{3}-\gamma k-2z_o$.
To have a large enough value, we must have $a=0$, $e=0$, $k=z_o=0$, and $U_\gamma\ge0$. Hence $\gamma\le(\sqrt{41}-5)/4<1/2$. The count inequality gives $z_h\le\gamma/(1-\gamma)<1$, so $z_h=0$ and the graph is $U$.
Finally, for $t\ge3$, using $D_2\ge3t^2/2-2t$ and $e\le\gamma(t-1)$ gives
\[
 P_\gamma\le-t^2+\frac{4t}{3}+2+
 \left(\frac{2t}{3}-1\right)(\gamma+\gamma^2)
 \le-t^2+\frac{8t}{3}\le-1.
\]
No other class remains.

\subsection*{A counting lemma for a balanced core}
At the endpoint we need the following counting lemma.

\begin{lemma}\label{lem:balanced-core-general}
Fix $r\ge3$ and $K\ge0$. Let
\[
 C=K_r\bigl((a+1)^{[j]},a^{[b]}\bigr),\qquad j+b=r,\quad b\ge1,
 \qquad P_0=(a+1)^j a^b.
\]
Add $k\le K$ internal edges, and then a vertex $v$ with exactly $h+k$ non-neighbors in the core, where $h\ge0$ is an integer with $h=o(a)$ and $q=a-h\ge1$. For sufficiently large $a$ the resulting graph satisfies $N_{K_{r+1}}(G)\ge q P_0/a$.
The conclusion is uniform along such sequences, including bounded $h$.
\end{lemma}
\begin{proof}
Put $c=P_0/a$. Let $w_H,w_L$ be the numbers of non-neighbors of $v$ in the larger and smaller core parts, so $w_H+w_L=h+k$. Let $k_H$ count internal edges in larger parts. Standard cliques entirely in the core number exactly $kc-\frac{k_Hc}{a+1}$.
These cliques are disjoint from the cliques formed by $v$ and one vertex from each core part. A union bound gives at least $qc+\frac{(w_H-k_H)c}{a+1}$
cliques of these two types together. Thus $w_H\ge k_H$ already proves the claim.

Suppose $w_H<k_H\le K$. If $w_H=0$, every internal edge in a larger part has both endpoints adjacent to $v$. Each such edge creates an additional family using $v$, that edge, and neighbors of $v$ from any $r-2$ other parts. These families are disjoint, and each has size at least $(r-1)(a-h-k)^{r-2}\ge c/(a+1)$
for large $a$. They are different from both types already counted. They compensate for the entire deficit $k_Hc/(a+1)$.

Assume next that $w_H\ge1$, and put $d=k_H-w_H\ge1$. Write $A_H,A_L$ for the proportions of the cliques with one vertex in each core part that meet at least one non-neighbor of $v$ in, respectively, a larger or a smaller part. Then
\[
 A_H\le\frac{w_H}{a+1},\quad A_L\le\frac{w_L}{a},\quad
 A_HA_L\ge\frac{w_Hw_L}{a(a+1)}(1-o(1)).
\]
For the last bound use $1-\prod_i(1-y_i)\ge\sum_i y_i-(\sum_i y_i)^2/2$, separately in the two groups. Here $w_H\le K$ and $w_L=o(a)$. The proportion avoiding all non-neighbors is $(1-A_H)(1-A_L)$. Keeping the overlap term improves the preceding lower bound to $N_{K_{r+1}}(G)\ge qc+ \frac{c}{a+1}\bigl(-d+w_Hw_L(1-o(1))\bigr)$.
We have $w_L=h+k-w_H\ge k-w_H\ge d$. If $w_Hw_L\ge d+1$, this proves the assertion for large $a$, since $d\le K$ is bounded. The only remaining possibility is $w_Hw_L=d$. It forces $w_H=1$, $w_L=d$, $h=0$, and $k=k_H=d+1$.

In that final case, the large-part avoidance factor is exactly $a/(a+1)$, and the small-part factor is at least $1-(k-1)/a$. The cliques through $v$ number at least $P_0(a-k+1)/(a+1)$. The standard core cliques number $kP_0/(a+1)$. Their sum is $P_0=qc$, because $h=0$. This proves the lemma in every case.
\end{proof}

\subsection*{Allowed neighborhoods on almost balanced cores}
Let $C_0=K_r((a+1)^{[j]},a^{[b]})$, where $j+b=r$ and $b\ge1$. Write $P_0=(a+1)^j a^b$ and $q=a-h$. Its two possible almost balanced cores are
\begin{align*}
 C_-&=K_r\bigl(a-1,a^{[b-2]},(a+1)^{[j+1]}\bigr) &&(b\ge2),\\
 C_+&=K_r\bigl(a+2,(a+1)^{[j-2]},a^{[b+1]}\bigr) &&(j\ge2).
\end{align*}
There are no internal edges or missing core cross-edges in this subsection. Add one vertex $v$ with exactly $h$ non-neighbors in the chosen core. We assume the strict inequality $N_{K_{r+1}}<qP_0/a$.

\begin{lemma}\label{lem:low-core}
Assume $0\le h\le a-2$. In $C_-$, no non-neighbor of $v$ may lie in an $(a+1)$-part. If there is no $a$-part, all $h$ non-neighbors lie in the $(a-1)$-part. If an $a$-part exists, the only possible distributions, up to a permutation of equal parts, are $X:(0,h),\qquad Y:(h-1,1),\qquad Z:(h,0)$,
where the entries count non-neighbors in the $(a-1)$-part and one $a$-part; all other parts have none. The distribution $Y$ requires $h\ge1$; $X=Y$ when $h=1$, and $X=Z$ when $h=0$. All these distributions satisfy the strict inequality.
\end{lemma}
\begin{proof}
The core product is $P_-=P_0(1-a^{-2})$. If at least one non-neighbor lies in larger parts, minimize the product of neighborhood sizes over all nonnegative non-neighbor counts for which the sum in the larger parts is at least one. All factors are positive. The logarithm is concave, so a minimum occurs at a vertex of this set. Its vertices put all $h$ non-neighbors in one larger part, or put one there and the other $h-1$ in one other part. Among the latter choices the smallest product puts the other non-neighbors in the $(a-1)$-part. These two products, minus $qP_0/a$, are respectively $P_0(h-1)/a^2$ and zero. They are nonnegative, contradicting the strict inequality. Thus no larger part contains a non-neighbor.

Let $u$ non-neighbors lie in the $(a-1)$-part. The others lie in $a$-parts. If no such part exists, $u=h$. Otherwise, for fixed $u$, the product of their avoidance factors is minimized by putting all $h-u$ non-neighbors in one $a$-part. Compared with $X$, this minimum increases the count by $\frac{P_0(a+1)}{a^3}\,u(h-1-u)$.
The count of $X$ is below the bound by $P_0(a-h)/a^3$. For $1\le u\le h-2$, we have $h\ge3$ and $u(h-1-u)\ge h-2\ge1$; hence the count exceeds the bound. This leaves $u=0,h-1,h$.

For $u=0$, a distribution among at least two $a$-parts has product at least the value at $(h-1,1)$. This follows by transferring integer mass to the largest entry until only one other positive entry, equal to one, remains; each transfer decreases the product. The resulting count minus the bound, after multiplication by $a^4/P_0$, is $(h-2)a^2+ha-h+1>0\qquad(h\ge2)$.
Thus $u=0$ requires concentration in one $a$-part. For $u=h-1$, only one further non-neighbor remains; $u=h$ leaves none. The product for $Y$ equals the product for $X$, and both are below the bound by $P_0(a-h)/a^3$. The product for $Z$ is below it by $P_0(h+1)/a^2$. This proves all assertions, including the small values of $h$.
\end{proof}

\begin{lemma}\label{lem:high-core}
Assume $0\le h\le a-2$. In $C_+$, the strict inequality forces all $h$ non-neighbors to lie in a single $a$-part. This distribution satisfies the strict inequality.
\end{lemma}
\begin{proof}
Here $P_+=P_0(a+2)a/(a+1)^2$. If a non-neighbor lies outside the $a$-parts, the same set argument as in Lemma~\ref{lem:low-core} applies. For a lower bound, replace the size of the chosen larger part by $a+1$ inside its avoidance factor, even if only the $(a+2)$-part is available. The two possible vertex types then both give a factor at least $(a+1-h)/(a+1)$. The difference between this lower bound and the required bound, multiplied by $a(a+1)^3/P_0$, is $a^2(h-1)+a(3h-1)+h>0\qquad(h\ge1)$.
Thus all non-neighbors lie in $a$-parts. If at least two of those parts contain them, moving the non-neighbors into one part gives the minimum at $(h-1,1)$. Its excess over the bound, multiplied by $a(a+1)^2/P_0$, is $a(h-2)+3h-2>0\qquad(h\ge2)$.
Only one $a$-part can therefore contain non-neighbors. The resulting count is $(qP_0/a)(1-(a+1)^{-2})$, which is strictly below the bound.
\end{proof}

\begin{lemma}\label{lem:core-order}
In the regime $h=o(a)$, the distributions of Lemma~\ref{lem:low-core} satisfy $\lam(X)>\lam(Y)>\lam(Z)$ when $h\ge2$, and $\lam(X)=\lam(Y)>\lam(Z)$ when $h=1$. If both cores exist, the graph from $C_+$ in Lemma~\ref{lem:high-core} has larger spectral radius than every graph satisfying the bound from $C_-$, for sufficiently large $a$.
\end{lemma}
\begin{proof}
Use the exact core response \eqref{eq:core-response}, and let $n=ra+j+1$ be the order including $v$. The roots are $\alpha n+O_r(1)$, and the derivative in $x$ is $-1/n+O_r(n^{-2})$. Put $A=x+a-1$ and write the extra response term as $\Psi_h(x,W,Q)=\frac{x(1-W)^2}{x(x+1)-h+Q}$.
Throughout the segment between the relevant values of $W,Q$,
\[
 (\Psi_h)_W=-\frac{2}{\alpha n}
 \left(1+O_r\left(\frac{h+1}{n}\right)\right),\qquad
 (\Psi_h)_Q=O_r(n^{-3}).
\]
For $Y$ minus $X$, the exact differences are
\[
 W_Y-W_X=\frac{h-1}{A(A+1)},\qquad
 Q_Y-Q_X=(h-1-2A)\frac{h-1}{A(A+1)}.
\]
The two changes have opposite signs, so both must be retained. Integrating the displayed partial derivatives, then using the root derivative, gives
\begin{equation}\label{eq:core-xy}
 \lam(X)-\lam(Y)=\frac{2(h-1)}{\alpha n^2}
 \left(1+O_r\left(\frac{h+1}{n}\right)\right).
\end{equation}
For $Z$ minus $Y$, $W_Z-W_Y=\frac1{A(A+1)},\qquad Q_Z-Q_Y=\frac{2h-1}{A}-\frac1{A+1}$.
The second difference is $O_r(hn)$ times the first. Thus the same argument gives
\begin{equation}\label{eq:core-yz}
 \lam(Y)-\lam(Z)=\frac{2}{\alpha n^2}
 \left(1+O_r\left(\frac{h+1}{n}\right)\right).
\end{equation}
The relative errors in \eqref{eq:core-xy} and \eqref{eq:core-yz} tend to zero for $h=o(a)$. The equalities at $h=0,1$ follow directly from the definitions.

If both cores exist and $C_-$ has an $a$-part, compare its best graph $X$ with the graph from $C_+$. Their $W,Q$ terms are identical. The core response difference is $g_x(a+2)-3g_x(a+1)+3g_x(a)-g_x(a-1)>0$
by \eqref{eq:third-difference} with the variable shifted. Hence the $C_+$ graph has strictly larger root. If $C_-$ has no $a$-part, its only allowed distribution is $Z$, with $W=h/(x+a-1)$ and $Q=h^2/(x+a-1)$. Both are at least the corresponding values for $C_+$, which use denominator $x+a$. The extra response decreases in $W$ when $W<1$ and in $Q$, and the core response is again strictly larger for $C_+$. This proves the last assertion, also at $h=0$.
\end{proof}

\subsection*{The endpoint}
Suppose a sequence of extremal graphs for the strict problem has $q/m\to1$. Corollary~\ref{cor:global-entry} gives $p\le n$ and $z=o(n)$, and Theorem~\ref{thm:core} gives a vertex $v$ meeting almost all internal edges. Write $C_0=T_{n-1,r}$ in the form $C_0=K_r((a+1)^{[j]},a^{[b]}),\quad j+b=r$,
where $a=m-1$, $j=r-1$, $b=1$ for $s=0$, and $a=m$, $j=s-1$, $b=r-s+1$ for $s>0$.
Let $\widehat h=m-q$, which is the parameter called $h$ in Lemma~\ref{lem:core-expansion}, and put
\[
 h=a-q=\widehat h-\chi_s,
 \qquad \chi_s=\begin{cases}1,&s=0,\\0,&s>0.\end{cases}
\]
In the partition of the core let $D_c$ be the squared part-size defect relative to $C_0$, let $k,z_o$ count its internal and missing edges, and set
$t=n-1-d_G(v)$ and $d_0=t-h$. Thus $d_0=\Delta+\chi_s$, where $\Delta=t-\widehat h$ is the variable in Lemma~\ref{lem:core-expansion}. These integer data and the part deviations are bounded. Equation~\eqref{eq:core-budget} gives
$\Delta\ge k-\chi_s-o(1)$, and hence integrality gives
\begin{equation}\label{eq:endpoint-integer-budget}
 d_0\ge k.
\end{equation}

There is an almost balanced core extension $R$ with $h$ non-neighbors that satisfies the strict inequality: use $C_+$ if available, and otherwise use $C_-$, with distribution $X$ if it has an $a$-part and $Z$ if not. Its squared defect is two. In terms of Section~\ref{sec:results}, this graph is $E_{n,r,q}$ in every other case and $V_{m,q}$ in the case $r=3,s=2$. The special one-edge choice is irrelevant here because $q\to\infty$.
For $R$ the actual number of non-neighbors is $h$, so its value of $\Delta$ is $-\chi_s$; for $G$ it is $d_0-\chi_s$. Thus their $\Delta$-values differ by exactly $d_0$. Applying \eqref{eq:core-spectrum} to $G$ and $R$ with the same comparison root gives
\begin{equation}\label{eq:general-endpoint-score}
 \begin{split}
 n(\lam(G)-\lam(R))
 &=-\alpha(D_c-2)+2k-2z_o-\frac{2d_0}{\alpha}+o(1)\\
 &\le 2\alpha-\alpha D_c-\frac{2k}{r-1}-2z_o+o(1).
 \end{split}
\end{equation}
The left side of \eqref{eq:general-endpoint-score} is nonnegative in the limit. As $2\alpha<2$, this excludes $z_o\ge1$, and it also excludes $D_c\ge4$. If $D_c=0$, then $k\le r-2$: indeed $k\le\alpha(r-1)=r-2+1/r$. Increasing $d_0$ from $k$ to $k+1$ makes the right side strictly negative, since $2/\alpha>2\alpha$. Thus the only remaining balanced-core data have $d_0=k$ and $z_o=0$. The graph is a balanced core with $k\le r-2$ internal edges, together with a vertex missing $h+k$ core vertices. Lemma~\ref{lem:balanced-core-general} gives at least $qP_0/a=q c_r(n)$ cliques, contradicting the strict inequality.

Therefore $D_c=2$. Equation \eqref{eq:general-endpoint-score} forces $k=z_o=d_0=0$. The core is exactly $C_-$ or $C_+$, and $v$ has exactly $h$ non-neighbors. Lemmas~\ref{lem:low-core}--\ref{lem:core-order} now identify the unique best graph. If $s=0$, only $C_+$ exists. If $s=1$, only $C_-$ exists and it has $r-2$ middle parts. If $s=2$, only $C_-$ exists and it has $r-3$ middle parts: its best graph is $X$ for $r\ge4$, and it is $Z=V$ for $r=3$. If $s\ge3$, $C_+$ exists and is strictly better than every allowed $C_-$ extension. These are exactly the graphs stated in Theorem~\ref{thm:main}.

\begin{proof}[Proof of the exact classification in Theorem~\ref{thm:main}]
For each $n$ and $q$, let $G$ be an extremal graph for the strict problem. If this conclusion failed for arbitrarily large orders with fixed $r$, choose a sequence witnessing failure and pass to a subsequence with fixed $s$ and convergent $q/m$. A zero limit is excluded by Subsection~\ref{sec:sublinear}, including the two small-$q$ exceptions. A limit in $(0,1)$ is excluded by the bounded comparisons in this section. A limit equal to one is excluded by the endpoint argument. In the case $r=3,s=2$ the only classes that remain for $q\ge4$ are $U,V$; their exact ordering is established in Section~\ref{sec:crossing}. All other cases have already been excluded, so these arguments prove uniqueness wherever a single class is stated, and exclude every other equality class. The conclusion is uniform in $q$. Indeed, if no single $n_0(r)$ worked, then for every integer $j$ there would be an order $n_j\ge j$, an admissible $q_j$, and a counterexample. After passing to a subsequence, $n_j\bmod r$ would be fixed and $q_j/\lfloor n_j/r\rfloor$ would converge to a point of $[0,1]$. The three cases just proved would then give a contradiction. Hence one value $n_0(r)$ works for every $1\le q<n/r$.
\end{proof}

\subsection{The non-strict problem}
\begin{proof}[Proof of Theorem~\ref{thm:closed}]
The graph $T^*_{n,r,q}$ has exactly $q c_r(n)$ copies of $K_{r+1}$ and has spectral radius $\alpha n+f_r(q/n)+O_r(n^{-1})$. By \cite[Theorem~7.2 and the equality statement following it]{FLLM}, the exact assertion and uniqueness already hold for $q\le\delta_r n$, where $\delta_r>0$ depends only on $r$. We therefore prove the exact assertion only for $q>\delta_r n$.

Suppose that a sequence of other extremal graphs exists in this remaining range. Passing to a subsequence, assume that $s$ is fixed and $q/m\to\gamma\in(0,1]$. Corollary~\ref{cor:global-entry} applies. If $0<\gamma<1$, Lemma~\ref{lem:integer} applies and its limit relative to $T^*$ must be nonnegative. Its upper bound is $-\alpha D_2+c_\gamma d-a_\gamma(1-\gamma)z_h-b_\gamma k-2z_o$. For $d<0$ this is negative. For $d>0$, $D_2\ge d(d+1)\ge2d$ and $c_\gamma<2\alpha$ again make it negative. Hence $d=0$, and all remaining nonpositive terms vanish. Equality in \eqref{eq:integer} also gives $e=0$, so the graph is the balanced graph with the prescribed star.

Suppose now that $\gamma=1$. Delete the vertex given by Theorem~\ref{thm:core}. Write the balanced core as $C_0=K_r((a+1)^{[j]},a^{[b]})$, put $h=a-q$, and define $\chi_s$ as in the strict endpoint argument. The parameter in Lemma~\ref{lem:core-expansion} is $\widehat h=h+\chi_s$. Let $D_c,k,z_o$ be the core defects and let $d_0=n-1-d_G(v)-h$. Since $d_0=\Delta+\chi_s$, equation~\eqref{eq:core-budget} and integrality give $d_0\ge k$.

For $T^*$ the core is balanced, there are no core edits, and the added vertex has exactly $h$ non-neighbors. The difference between its value of $\Delta$ and that of $G$ is $d_0$. Subtracting the two core expansions gives
\[
 n(\lam(G)-\lam(T^*))
 =-\alpha D_c+2k-2z_o-\frac{2d_0}{\alpha}+o(1)
 \le-\alpha D_c-\frac{2k}{r-1}-2z_o+o(1).
\]
After the bounded integer parameters are fixed, nonnegativity forces $D_c=k=z_o=d_0=0$. The core is balanced and the added vertex has exactly $h=o(a)$ non-neighbors. Its clique count is $\prod_i(b_i-w_i)$. Since $\sum_i\log(b_i-w_i)$ is strictly concave on $w_i\ge0$, $\sum_iw_i=h$, its minimum is attained by placing all non-neighbors in one smallest part. The value is $qP_0/a=q c_r(n)$, so the non-strict inequality forces precisely this distribution. Thus $G=T^*$.

These arguments, together with the cited small-linear-range result, prove the exact statement uniformly for all $1\le q<n/r$. It remains to prove \eqref{eq:closed-gap}. Suppose that for some fixed $\eta>0$ there are graphs outside $T^*_{n,r,q}$, with $q/m\ge\eta$, such that $n(\lam(T^*_{n,r,q})-\lam(G))\to0$. Corollary~\ref{cor:global-entry} and the local reductions apply with a fixed error in the spectral lower bound. After taking a subsequence, $q/m$ has a limit $\gamma\in[\eta,1]$. If $\gamma<1$, the interior comparison above applies; if $\gamma=1$, the endpoint comparison applies. Both arguments require only a nonnegative limiting spectral difference and force every defect parameter to vanish. Hence the graph is $T^*$, contrary to the choice of the sequence. This proves the uniform gap.
\end{proof}

\section{Crossing points and thresholds}\label{sec:crossing}
\subsection*{A crossing in the tripartite case}

For real $z\in[1,m]$, equation \eqref{eq:star-response} defines two functions $\lam_U(z)$ and $\lam_V(z)$. When the other two parts both have size $v$, the root equation is $R_{\rm st}(x;u,p)=(x-v)/(2v)$. Clearing denominators gives $P(x;u,v,p)=x^4-vx^3-(p+2uv)x^2-3vpx+2vp(u-p-1)=0$.
Thus the two roots satisfy
\begin{align*}
 P_U(x,z)&=x^4-mx^3-(z+2m(m+2))x^2-3mzx+2mz(m+1-z),\\
 P_V(x,z)&=x^4-(m+1)x^3-(z-1+2m(m+1))x^2\\
 &\hspace{12mm}-3(m+1)(z-1)x+2(m+1)(z-1)(m-z).
\end{align*}
Put $\gamma=z/m$. Uniformly on $[1,m]$,
\begin{equation}\label{eq:crossing-expansion}
 \lam_U(z)-\lam_V(z)=\frac{A(\gamma)}m+\frac{B(\gamma)}{m^2}+O(m^{-3}),
\end{equation}
where
\[
 A(\gamma)=\frac{2-5\gamma-2\gamma^2}{9},\qquad
 B(\gamma)=\frac{22\gamma^4+100\gamma^3+175\gamma^2+152\gamma+64}{216}.
\]
We include a coefficient verification. Set $\epsilon=1/m$, $y=x/m$ and $P_0(y)=y^4-y^3-2y^2$, so $P_0'(2)=12$. The scaled polynomials are $P_0+\epsilon P_1+\epsilon^2P_2+\epsilon^3P_3$. For $U$,
\[
 P_{1,U}=-(\gamma+4)y^2-3\gamma y+2\gamma(1-\gamma),\quad
 P_{2,U}=2\gamma,\quad P_{3,U}=0.
\]
For $V$,
\begin{align*}
 P_{1,V}&=-y^3-(\gamma+2)y^2-3\gamma y+2\gamma(1-\gamma),\\
 P_{2,V}&=y^2-3(\gamma-1)y-2(\gamma-1)^2,\\
 P_{3,V}&=3y-2(1-\gamma).
\end{align*}
Writing $x=2m+h+k/m+j/m^2+O(m^{-3})$, coefficient comparison gives
\begin{align*}
 12h+P_1(2)&=0,\\
 12k+16h^2+P_1'(2)h+P_2(2)&=0,\\
 12j+32hk+7h^3+P_1'(2)k+\tfrac12P_1''(2)h^2
 +P_2'(2)h+P_3(2)&=0.
\end{align*}
These identities give $h=(\gamma^2+4\gamma+8)/6$, $k_U-k_V=A(\gamma)$ and $j_U-j_V=B(\gamma)$. The simple root $y=2$ and closed interval $\gamma\in[0,1]$ ensure that the expansions from the implicit function theorem and their $\gamma$ derivatives are uniform. In particular,
\begin{equation}\label{eq:crossing-derivative}
 \frac{d}{dz}\bigl(\lam_U(z)-\lam_V(z)\bigr)
 =-\frac{5+4\gamma}{9m^2}+O(m^{-3})<0.
\end{equation}
The difference has opposite signs at $z=1,m$; its strict monotonicity in \eqref{eq:crossing-derivative} gives exactly one crossing. The leading root is $\gamma_0=(\sqrt{41}-5)/4$. Substituting $z=\gamma_0m+b$ into \eqref{eq:crossing-expansion} gives coefficient $A'(\gamma_0)b+B(\gamma_0)$ at $m^{-2}$. Reduction by $2\gamma_0^2+5\gamma_0-2=0$ gives
\[
 A'(\gamma_0)=-\frac{\sqrt{41}}9,\qquad
 B(\gamma_0)=\frac{198-19\gamma_0}{288}.
\]
Solving for $b$ proves \eqref{eq:crossing-location}; the remaining error is $O(m^{-3})$ and the slope has order $m^{-2}$.

This proves the crossing assertion in Theorem~\ref{thm:main}. Its use is restricted there to $q\ge4$; the mixed graph $M_m$ has already been treated separately at $q=3$.

\subsection*{The difference between the strict and non-strict problems}
Put $J_r(n,q)=\lam(T^*_{n,r,q})-\Lambda_r^<(n,q)$.
The two integer bounds in this difference are $q c_r(n)$ and $q c_r(n)-1$.

\begin{corollary}\label{cor:jump}
For fixed $r\ge3$, uniformly over $1\le q<n/r$, one has
\begin{equation}\label{eq:regular-jump}
 J_r(n,q)=\frac{2(r-1)}{rn}+O_r(n^{-2})
 \qquad\text{if }(r,n\bmod r)\ne(3,2).
\end{equation}
For $n=3m+2$, putting $\gamma=q/m$ gives
\begin{equation}\label{eq:exception-jump}
 J_3(n,q)=\frac1n\min\left\{
 \frac{4+4\gamma+\gamma^2}{3},\frac{6-\gamma-\gamma^2}{3}
 \right\}+O(n^{-2}).
\end{equation}
In particular, the jump is uniformly of order $n^{-1}$ for every fixed clique in the entire natural range.
\end{corollary}
\begin{proof}
In every other case, outside the small exception $s=2,q=1$, the strict and closed graphs have the same star size in parts of the same size. Their changes in the response cancel exactly. Their base squared defects differ by two, so at $x=\alpha n+O(1)$ their responses differ by $2\alpha/n^2+O_r(n^{-3})$, uniformly for $q\le m$. Division by the root derivative $1/n+O_r(n^{-2})$ proves \eqref{eq:regular-jump}. For $s=2,q=1$, moving the star part by one changes the comparison only by $O_r(n^{-2})$.

For the case $r=3$ and $s=2$, the bounded response expansion used in Lemma~\ref{lem:integer} can be applied directly to these two fixed configurations, uniformly for $0\le\gamma\le1$. Relative to $T^*$, the data for $U$ are $(D_2,d,\ell,k,z_h,z_o)=(2,1,0,0,0,0)$,
and hence
\[
 n\{\lam(U)-\lam(T^*)\}=-\frac43-\frac{4\gamma+\gamma^2}{3}+O(n^{-1}).
\]
For $V$ the data are $(D_2,d,\ell,k,z_h,z_o)=(0,-1,-1,0,0,0)$,
so
\[
 n\{\lam(V)-\lam(T^*)\}=-(2+\gamma)+\frac{4\gamma+\gamma^2}{3}+O(n^{-1}).
\]
Equivalently, $n(\lam(T^*)-\lam(U))=\frac{4+4\gamma+\gamma^2}{3}+O(n^{-1})$,
\[
 n(\lam(T^*)-\lam(V))=\frac{6-\gamma-\gamma^2}{3}+O(n^{-1}).
\]
The uniformity also follows by expanding the two rational star responses on the interval $0\le\gamma\le1$. For $q\ge4$, Theorem~\ref{thm:main} takes the larger of $U,V$, proving \eqref{eq:exception-jump}. For $q=1,2$, that larger graph is $U$ for large $m$. At $q=3$, replacement of $U$ by $M_m$ changes the jump by $2/n^2+O(n^{-3})$, already within the stated error by \eqref{eq:mixed-gap}. Finally both entries in the minimum have a positive lower bound on $[0,1]$, proving the order assertion.
\end{proof}

\subsection*{The matching threshold}
We now recover the matching threshold on the larger range given by the exact strict theorem, and retain the constant term in its crossing. Put $\alpha=(r-1)/r$ and
\[
 \kappa_{r,s}=\begin{cases}1,&(r,s)=(3,2),\\0,&\text{otherwise}.
 \end{cases}
\]
For $q=\Theta(\sqrt n)$, let $E^{\rm st}_{n,r,q}$ denote the strict star graph: it is $E_{n,r,q}$ in every other case and $U_{m,q}$ in the case $r=3,s=2$. The strict optimum equals this branch throughout any fixed positive square-root window for sufficiently large $n$.

For real $z$ in such a window and a special part of size $u$, the exact matching and star responses are
\[
 R_M(x;u,z)=\frac{u-2z}{x}+\frac{2z}{x-1},\qquad
 R_S(x;u,z)=\frac{u-z-1}{x}+\frac{(z+1)x+2z}{x^2-z}.
\]
Both define roots of the response equations, agreeing with graph spectra at integer $z$. Their increments over $u/x$ are
\[
 R_M-\frac ux=\frac{2z}{x(x-1)},\qquad
 R_S-\frac ux=\frac{z(2x+z+1)}{x(x^2-z)}.
\]
The relevant roots satisfy $x=\alpha n+O_r(n^{-1/2})$. The base of $E^{\rm st}$ has squared part-size defect two; its star part is larger than the matching part by $\kappa_{r,s}$. Expansion of the exact responses, or \eqref{eq:fine} at integer parameters, yields
\begin{equation}\label{eq:matching-difference}
 \lam_{E^{\rm st}}(z)-\lam_Y(z)
 =-\frac{2\alpha}{n}+\frac{z^2-z}{\alpha n^2}
 -\frac{4\kappa_{r,s}z}{n^2}+O_r(n^{-2}).
\end{equation}
Indeed, the squared-degree sums of a $z$-edge star and matching are $z^2+z$ and $2z$, respectively. The remainder in \eqref{eq:fine} is $O((z^2+z)/n^3)=O(n^{-2})$, and the base expansion has the same error order.

We also need to control the derivative. Set $\epsilon=n^{-1/2}$ and $t=z\epsilon$. After writing $x/n=\alpha+u$, the rational response equations are analytic in $(\epsilon,t,u)$ near $\epsilon=0$, for $t$ in any fixed closed positive interval. At $\epsilon=0,u=0$, their derivative in $u$ is $-1$. The implicit function theorem therefore makes the expansion uniform together with its $t$ derivative. The displayed increments give
\begin{equation}\label{eq:matching-derivative}
 \frac{d}{dz}\bigl(\lam_{E^{\rm st}}(z)-\lam_Y(z)\bigr)
 =\frac{2z}{\alpha n^2}+O_r(n^{-2}).
\end{equation}
Here we differentiate the rational equations themselves, so the error is uniform after differentiation.

\begin{proof}[Proof of Theorem~\ref{thm:matching-threshold}]
For $z=\delta\sqrt n$, multiply \eqref{eq:matching-difference} by $n$. Its limit is $-2\alpha+\delta^2/\alpha$, whose unique positive zero is $\sqrt2\alpha$. Moreover, for every fixed $b>0$, the same comparison is uniform for integers $1\le q\le b\sqrt n$: the small-$q$ exceptions in Theorem~\ref{thm:main} change the difference by only $O_r(n^{-2})$. If $\delta<\sqrt2\alpha$, the strict maximum is therefore below $\lam(Y_{n,r,q})$ for all integers $q<\delta\sqrt n$ and sufficiently large $n$. This proves the desired implication. If $\delta>\sqrt2\alpha$, choose $\delta'$ strictly between them and put $q=\lfloor\delta'\sqrt n\rfloor$. The strict extremal star has larger spectrum than the matching graph and violates the implication. Thus the supremum is $\delta_r^*=\sqrt2\alpha$.

In a window $a\sqrt n\le z\le b\sqrt n$ with $0<a<\sqrt2\alpha<b$, the difference has opposite signs at the endpoints, and \eqref{eq:matching-derivative} is strictly positive for large $n$. There is exactly one crossing. Substitute $z=\sqrt2\alpha\sqrt n+B$ into \eqref{eq:matching-difference}. The coefficient of $n^{-3/2}$ is $\sqrt2\alpha\left(\frac{2B-1}{\alpha}-4\kappa_{r,s}\right)$.
It vanishes at $B=1/2+2\alpha\kappa_{r,s}$. The remaining spectral error is $O(n^{-2})$, while the slope is $\Theta(n^{-3/2})$. The location error is therefore $O(n^{-1/2})$, proving \eqref{eq:matching-location}. At an integer crossing the strict extremal graph itself has spectrum equal to the matching threshold, so the implication with a weak spectral inequality fails there.
\end{proof}
\section{Further results}\label{sec:gaps}

\subsection{Comparing complete multipartite graphs}
The exact comparison \eqref{eq:third-difference} is a special case of the following moment expansion. It is a direct consequence of Zhang's multipartite identity \cite{ZhangWalks}; we include the short calculation because the coefficient and the uniform error are used below.

\begin{proposition}\label{thm:moments}
Fix $r\ge3$, an integer $k\ge2$, and positive constants $A,B$. Let $G_a,G_b$ have a common chosen part of size $u=n/r+a_0$, containing the same graph $H$ with at most $Bn$ edges. All other parts are independent, and all cross-edges are present. Their part deviations from $n/r$ are $a_1,\ldots,a_{r-1}$ and $b_1,\ldots,b_{r-1}$, with absolute values at most $A$. Assume the graphs have the same order and $\sum_{i=1}^{r-1}a_i^j=\sum_{i=1}^{r-1}b_i^j \qquad(1\le j<k)$.
Writing $\Delta M_k=\sum_i a_i^k-\sum_i b_i^k$, we have
\begin{equation}\label{eq:moment-comparison}
 \lam(G_a)-\lam(G_b)
 =(-1)^{k-1}\al_r\frac{\Delta M_k}{n^{k-1}}
 +O_{r,k,A,B}(n^{-k}).
\end{equation}
\end{proposition}
\begin{proof}
Apply Zhang's multipartite response identity in resolvent form. The chosen-part responses cancel. For $g_x(t)=t/(x+t)$, $\frac{g_x^{(j)}(t)}{j!}=(-1)^{j-1}\frac{x}{(x+t)^{j+1}}$.
Expand at $t=n/r$ and use the assumed moment equalities. At $x=\al_r n+O(1)$,
\[
 F_a(x)-F_b(x)
 =(-1)^{k-1}\frac{x\Delta M_k}{(x+n/r)^{k+1}}
 +O(n^{-k-1}).
\]
Both roots lie in this interval by \eqref{eq:fine}, and the response derivatives are $-1/n+O(n^{-2})$. The mean-value theorem proves \eqref{eq:moment-comparison}.
\end{proof}

For $3\le s\le r-2$, the two stars on the two almost balanced bases with their stars in an $(m+1)$-part have equal first and second part-size moments. Their third moments differ by six. Thus their spectra separate at order $n^{-2}$, with coefficient $6(r-1)/r$, rather than at the order $n^{-1}$ determined by a difference in the quadratic part-size defect.

\subsection{Gaps in the sublinear range}
Let $W_{n,r,q}=E_{n,r,q}$ in every other case, and let $W_{n,3,q}=U_{m,q}$ when $s=2$. The next statement concerns $q$ bounded below by a sufficiently large constant, so neither small-$q$ exception occurs.

\begin{theorem}\label{thm:small-range-gap}
For every fixed $r\ge3$, there are constants $Q_r,\varepsilon_r,\zeta_r>0$ such that, for all sufficiently large $n$ and all integers $Q_r\le q\le\varepsilon_r n$, the unique extremal graph is $W_{n,r,q}$. If $G\not\cong W_{n,r,q}$ also satisfies the strict inequality, then
\begin{equation}\label{eq:small-range-gap}
 \lam(W_{n,r,q})-\lam(G)\ge
 \begin{cases}
 \zeta_r n^{-2},&3\le s\le r-2,\\
 \zeta_r qn^{-2},&\text{otherwise}.
 \end{cases}
\end{equation}
Both orders are best possible. When $3\le s\le r-2$, the star extension $W^-$ of $B^-_{n,r}$ with its star in an $(m+1)$-part is the graph with the second largest spectral radius, up to isomorphism, after increasing $Q_r$ if necessary, and
\begin{equation}\label{eq:second-base-gap}
 \lam(W_{n,r,q})-\lam(W^-)
 =\frac{6\al_r}{n^2}+O_r(n^{-3}).
\end{equation}
\end{theorem}
\begin{proof}
Choose $L=L(r)$ so that $\lam(W)\ge\al_r n+f_r(q/n)-L/n$. If $\lam(G)\ge\lam(W)-1/n$, then Corollary~\ref{cor:global-entry} and Proposition~\ref{thm:reduction}, used with $\tau=1/2$, apply once $\varepsilon_r$ is small enough. Hence the part-size differences, the number of missing cross-edges, and $p-q$ are bounded by constants depending only on $r$.

The proof of Lemma~\ref{lem:second-layer} gives the following alternatives. Standard completions rule out $p\ge q+1$ once $\varepsilon_r$ is sufficiently small. If $p\le q-1$, $z\ge1$, or the base squared defect is at least four, comparison \eqref{eq:strict-coarse} loses at least $c_0/n$: its error, multiplied by $n$, is $O_r(\varepsilon_r+\varepsilon_r^2+n^{-1})$. A balanced base with $p=q,z=0$ does not satisfy the strict inequality. Graphs below the comparison range already have a gap of at least $1/n$. Thus only $q$-edge extensions of $B^-$ or $B^+$ with all cross-edges present require a finer comparison.

The allocation arguments in Subsection~\ref{sec:sublinear} are uniform in this window after shrinking $\varepsilon_r$. In particular, their uses of $q=o(m)$ require only fixed sufficiently small bounds on $q/m$. A graph other than one star loses at least $2(q-3)$ in its squared-degree sum. A star in a larger allowed part loses a fixed multiple of $q/n^2$ in the part-size term. When $s=2$, \eqref{eq:s2-allocation} gives the same order for the losing allocation or base. The relative error from \eqref{eq:fine} is $O_r(q/n+n^{-1})$ compared with these $q/n^2$ losses. Choose $Q_r$ large and $\varepsilon_r$ small so that these losses remain at least $c_1q/n^2$.

The only smaller comparison occurs when $3\le s\le r-2$. Here the stars on the two bases have the same part size, and their internal responses cancel exactly in \eqref{eq:third-difference}. Their roots are $\al_rn+O_r(1)$, uniformly in the linear range, so the response derivative gives \eqref{eq:second-base-gap}. This proves the lower bounds in \eqref{eq:small-range-gap}. In the case $r=3,s=2$ take $\varepsilon_r$ still smaller so that $q/m$ lies below half the limiting crossing value; $U$ is then the unique extremal graph throughout this window.

For sharpness when $3\le s\le r-2$, use $W^-$. In every other case, replace the internal star by the disjoint union $K_{1,q-1}\cup K_2$, adding isolated vertices to fill its part. It fits for our choice of $\varepsilon_r$. This internal graph is triangle-free and lies in one part, so its clique count is exactly that of $W$. Its squared-degree sum is $q^2-q+2$, and \eqref{eq:fine} gives a loss $\frac{2(q-1)}{\al_r n^2}+O_r((q^2+q)/n^3)=O_r(q/n^2)$.
Finally, when $3\le s\le r-2$, every other graph except $W^-$ loses at least $c_1q/n^2$ or $c_0/n$. Increasing $Q_r$ makes both bounds larger than \eqref{eq:second-base-gap}, proving the claim about the second largest graph.
\end{proof}

In particular, \eqref{eq:small-range-gap} holds whenever $q\to\infty$ and $q=o(n)$. If $q=\Theta(\sqrt n)$, the bound in the second line is of order $n^{-3/2}$. Thus there are two different sources of a spectral loss. Pairs of disjoint internal edges measure the loss when the internal graph is not a star, while the first unequal part-size moment measures the smaller difference between the two bases.

\subsection{Gaps for tripartite graphs}
\begin{proposition}\label{prop:tripartite-gaps}
Fix $0<\eta<1$. For all sufficiently large $m$, the following statements hold.
For $n=3m$ and $\eta m\le q\le m-1$, every graph not isomorphic to $E_{n,3,q}$ loses at least $c_\eta/n$ in spectral radius.
For $n=3m+1$, the same statement for graphs not isomorphic to $E_{n,3,q}$ holds on $\eta m\le q\le(1-\tau)m$ for each fixed $\tau>0$, with a constant allowed to depend on $\tau$.
For $n=3m+2$ and $\eta m\le q\le m$, every graph not isomorphic to $U_{m,q}$ or $V_{m,q}$ satisfies
\begin{equation}\label{eq:tripartite-gap}
 \LL-\lam(G)\ge c_\eta/n.
\end{equation}
\end{proposition}
\begin{proof}
If one of these bounds failed, take graphs other than the listed ones for which the gap multiplied by $n$ tends to zero. Corollary~\ref{cor:global-entry} and the two reduction results apply with a fixed error in the spectral lower bound. Pass to a subsequence on which $q/m$ has a positive limit. A limit less than one is excluded by the integer comparisons in Subsection~\ref{sec:terminal}, except for the listed graphs. If the limit is one, the endpoint proof gives an almost balanced core with no remaining edits. For $n=3m$, only the core $C_+$ is possible and its best neighborhood gives the single graph listed above. For $n=3m+2$, the core $C_-$ has no middle part; its only allowed neighborhood is $V$. Thus the endpoint also gives a contradiction in those two cases. The case $s=1$ at the endpoint is excluded from this proposition and is treated below. The same argument gives positive constants independent of $q$.
\end{proof}

\begin{corollary}\label{cor:linear-crossing-gap}
For $n=3m+2$, let $d=q-q_*(m)$ and suppose $|d|=o(m)$. Then
\begin{equation}\label{eq:local-linear-gap}
 |\lam(U_{m,q})-\lam(V_{m,q})|
 =\frac{\sqrt{41}}{9m^2}|d|
 \left(1+O\left(\frac{|d|+1}{m}\right)\right).
\end{equation}
For large $m$, these are the two graphs with largest spectral radii, up to isomorphism, among all graphs satisfying the strict inequality; at a tie every other graph is at least $c/n$ below them.
\end{corollary}
\begin{proof}
Integrate \eqref{eq:crossing-derivative} from the exact root $q_*(m)$ to $q$. Since $q_*(m)/m=(\sqrt{41}-5)/4+O(m^{-1})$, the derivative is $-\sqrt{41}/(9m^2)+O((|d|+1)/m^3)$ throughout this interval. This proves \eqref{eq:local-linear-gap}. The graph gap is $o(n^{-1})$, whereas \eqref{eq:tripartite-gap} separates every other graph by a fixed multiple of $n^{-1}$.
\end{proof}

\subsection{The three largest classes near an endpoint}
Let $n=3m+1$, $q=m-h$, and start with the core $K_3(m-1,m,m+1)$. Denote by $E_{m,h}(a,b,c)$ the graph obtained by adding one vertex with $a,b,c$ non-neighbors in the respective core parts, where $a+b+c=h$. Define
\[
 X_{m,h}=E_{m,h}(0,h,0),\quad
 Y_{m,h}=E_{m,h}(h-1,1,0),\quad
 Z_{m,h}=E_{m,h}(h,0,0).
\]
The middle graph is defined only for $h\ge1$. The notation $Y_{m,h}$ in this subsection refers to this core extension, not to the matching graph $Y_{n,r,q}$. We have $X_{m,h}=E_{n,3,q}$ and $Z_{m,h}=S(m,m,m+1;q-1)$. Lemma~\ref{lem:low-core} shows that these are the only neighborhoods satisfying the strict inequality on this core, up to a permutation of equal parts. Their counts are
\[
 N_{K_4}(X)=N_{K_4}(Y)=q(m^2-1),\qquad
 N_{K_4}(Z)=(q-1)m(m+1)<qm^2.
\]

\begin{theorem}\label{thm:endpoint-order}
There are constants $\delta,c>0$ and $m_0$ such that, for $m\ge m_0$ and $2\le h\le\delta m$, the graphs $X_{m,h},Y_{m,h},Z_{m,h}$, in that order, have the three largest spectral radii up to isomorphism among all graphs with $N_{K_4}(G)<qm^2$. Every other graph $G$ satisfies
\begin{equation}\label{eq:endpoint-separated}
 \lam(X_{m,h})-\lam(G)\ge c/n.
\end{equation}
Uniformly in this interval,
\begin{align}
 \lam(X_{m,h})-\lam(Y_{m,h})
 &=\frac{3(h-1)}{n^2}\left(1+O\left(\frac{h+1}{n}\right)\right),\label{eq:endpoint-xy}\\
 \lam(Y_{m,h})-\lam(Z_{m,h})
 &=\frac3{n^2}\left(1+O\left(\frac{h+1}{n}\right)\right).\label{eq:endpoint-yz}
\end{align}
At $h=1$, $X=Y$ and $Z$ has the second largest spectral radius, with \eqref{eq:endpoint-yz} still valid. At $h=0$, $X=Z$; every other graph has a gap of at least $c/n$.
\end{theorem}
\begin{proof}
We first prove the gap from all other graphs. If no fixed $\delta,c,m_0$ worked, there would be a sequence with $h/m\to0$ and graphs other than the listed ones satisfying $n(\lam(X)-\lam(G))\to0$. Corollary~\ref{cor:global-entry} and Theorem~\ref{thm:core} apply. After passing to a subsequence, all bounded integer parameters of the core are fixed. The proof of \eqref{eq:general-endpoint-score}, specialized to $r=3,s=1$, shows that a nonnegative limit excludes all core defects at least four and all missing core edges. For a balanced core, only $k=0,1$ and $d_0=k$ remain; Lemma~\ref{lem:balanced-core-general} excludes both. Thus the graph is obtained from an almost balanced core by adding one vertex with exactly $h$ non-neighbors. This core is $K_3(m-1,m,m+1)$, and Lemma~\ref{lem:low-core} leaves exactly $X,Y,Z$. This contradiction proves \eqref{eq:endpoint-separated} in a fixed endpoint window, including the equalities at $h=0,1$.

The response calculation in Lemma~\ref{lem:core-order} gives \eqref{eq:endpoint-xy}--\eqref{eq:endpoint-yz} with $\al_3=2/3$. Its estimates are uniform also for $h\le\delta m$ once $\delta$ is sufficiently small: the roots remain $2n/3+O(1)$, $W=O(h/n)$, the partial derivative in $W$ is $-3n^{-1}(1+O((h+1)/n))$, the partial derivative in $Q$ is $O(n^{-3})$, and the full root derivative is $-1/n+O(n^{-2})$. The exact differences of $W,Q$ displayed there therefore have the stated relative errors. Shrink $\delta$ so that both errors have absolute value less than one half. This proves the strict order for $h\ge2$ and the order at $h=1$.

Finally, shrink $\delta$ again so that the sum of the two displayed gaps is less than $c/(2n)$. Then \eqref{eq:endpoint-separated} puts every other graph below $Z$. Thus these are the three graphs with largest spectral radii. The equalities at $h=0,1$ follow from the actual neighborhoods and describe the same graph.
\end{proof}

\subsection{An example near the endpoint}
For a graph $G$, let
\[
 z_3(G)=\min_{(P_1,P_2,P_3)}
 \sum_{i<j}\bigl(|P_i||P_j|-e_G(P_i,P_j)\bigr),
\]
where the minimum ranges over partitions into three nonempty parts. This quantity counts missing cross-edges only and allows arbitrary internal edges in the new partition.

\begin{proposition}\label{prop:core-example}
Choose integers $h\to\infty$ with $h=o(\sqrt m)$, put $q=m-h$, and set $a=\lfloor(h+1)/2\rfloor,\qquad b=h+1-a$.
Start with $K_3(m+1,m,m)$ on parts $A,B,C$. Add a vertex $v$ adjacent to every core vertex except $a$ vertices of $A$ and $b$ vertices of $B$, and call the graph $W_{m,h}$. Then, for sufficiently large $m$,
\begin{align}
 N_{K_4}(W_{m,h})-qm(m+1)
 &=m\bigl(-m+a(b+1)-1\bigr)<0,\label{eq:example-count}\\
 \lam(V_{m,q})-\lam(W_{m,h})
 &=\frac5{3n}+o(n^{-1}),\qquad n=3m+2,\label{eq:example-gap}\\
 z_3(W_{m,h})&=\min\{a,b\}\longrightarrow\infty.\label{eq:example-deletions}
\end{align}
However, deleting $v$ leaves a balanced complete tripartite graph.
\end{proposition}
\begin{proof}
The core contains no $K_4$, so $N_{K_4}(W)=(m+1-a)(m-b)m$. Using $a+b=h+1$ gives \eqref{eq:example-count}, whose sign follows from $a(b+1)=O(h^2)=o(m)$.

The core squared deviation of $W$ is $2/3$, whereas the core of $V$ has squared deviation $8/3$. Their added vertices have $h+1$ and $h$ non-neighbors, respectively. There are no core edits. Lemma~\ref{lem:core-expansion} gives
\[
 n(\lam(W)-\lam(V))
 =-\frac23\left(\frac23-\frac83\right)-3+o(1)
 =-\frac53+o(1),
\]
which proves \eqref{eq:example-gap} independently of the exact extremal theorem.

Assigning $v$ to the core part containing more of its non-neighbors gives a partition with $\min\{a,b\}$ missing cross-edges. Conversely, if any core part is split by a new partition into three parts, its own non-edges across that split already contribute at least $m-1>\min\{a,b\}$ missing cross-edges. Hence a partition with a smaller count would keep every core part intact. If the three core parts occupy three different new parts, placing $v$ can absorb at most $\max\{a,b\}$ of its non-neighbors. At least $\min\{a,b\}$ remain across parts. If the core occupies two new parts, the third consists of $v$ and all $h+1$ non-neighbors are counted. The core cannot occupy only one part, because a single remaining vertex cannot fill two further nonempty parts. This proves \eqref{eq:example-deletions}.
\end{proof}

In its original partition with $v$ assigned to $A$, this example has $O(n)$ internal edges and $o(n)$ missing cross-edges. It shows that even after optimizing the partition, a bounded missing-edge conclusion is too strong for all graphs within $O(n^{-1})$ of the strict optimum. The bounded one-vertex core in Theorem~\ref{thm:core} remains valid and gives the right form near the endpoint.

In all exact results, $r$ is fixed and $n$ is sufficiently large. Theorem~\ref{thm:main} includes every integer $1\le q<n/r$, with its two explicitly stated small-$q$ exceptions. Theorem~\ref{thm:global-asymptotic} gives a first-order result on larger fixed linear ranges; it is not an exact classification for $q\ge n/r$. The gap estimates hold in the ranges stated in the theorems, and the equalities at the endpoint in Theorem~\ref{thm:endpoint-order} refer to the same graph, not to different extremal graphs.

\section*{Data availability}

No data were used for the research described in this article.

\section*{Declaration on the use of generative AI}

The authors used ChatGPT (OpenAI) to discuss proof strategies, check intermediate arguments, and improve the exposition. The authors reviewed and verified the mathematical content and take full responsibility for the manuscript.

\end{document}